\newif\iflncs
\lncsfalse

\newif\ifnotes\notesfalse

\documentclass[11pt,a4paper]{article}
\usepackage{multicol}
\usepackage[margin=2cm]{geometry}
\usepackage[font=small]{caption}
\usepackage{psfrag}
\usepackage{amssymb,latexsym, bm, dsfont,graphicx}
\usepackage{amsthm,amsfonts, amsmath}
\usepackage{enumerate}
\usepackage{algorithm}
\usepackage{stmaryrd}
\usepackage{float}
\usepackage{hyperref}
\usepackage[noend]{algpseudocode}
\usepackage{subfigure,appendix}
\usepackage{thmtools, thm-restate}
\usepackage{tabularx}   
\usepackage{nicefrac}
\usepackage[noend]{algpseudocode}
\usepackage{url}

\urldef{\mailsa}\path|dadush@cwi.nl,|
\urldef{\mailsb}\path|{l.vegh, g.zambelli}@lse.ac.uk,|

\newtheorem{theorem}{Theorem}[section]
\newtheorem{corollary}[theorem]{Corollary}

\newtheorem{lemma}[theorem]{Lemma}

\newtheorem{Claim}[theorem]{Claim}

\newcommand{\bb}{\mathbb}
\newcommand{\R}{\bb R}
\newcommand{\Z}{\bb Z}

\newcommand{\E}{\bb E}
\newcommand{\B}{\bb B}

\newcommand{\ceil}[1]{\lceil#1\rceil}

\newcommand{\sm}{\setminus}
\newcommand{\smz}{\setminus\{0\}}

\newcommand{\scalar}[1]{\left\langle #1\right\rangle}
\newcommand{\set}[1]{\left\{ #1 \right\}}
\newcommand{\spn}{\mathrm{span}}
\newcommand{\supp}{\mathrm{supp}}

\newcommand{\T}{\mathsf{T}}

\newcommand{\bart}{\theta}

\newcommand{\qued}{}

\DeclareMathOperator{\width}{\mathrm{width}}
\DeclareMathOperator{\conv}{\mathrm{conv}}

\DeclareMathOperator{\rk}{rk}
\DeclareMathOperator{\im}{im}
\DeclareMathOperator{\tr}{tr}
\DeclareMathOperator{\vol}{vol}

\def\const{\frac 1{11m}}  

\def\st{\,:\,}
\def\eps{\varepsilon}

\newcommand{\eqdef}{\mathbin{\stackrel{\rm def}{=}}}

\floatstyle{plain}
\newfloat{myalgo}{tbhp}{mya}

\newcommand\ourpagewidth{.9\textwidth}

{\begin{myalgo}[#1]
\centering
\begin{minipage}{\ourpagewidth}
\begin{algorithm}[H]}%
{\end{algorithm}
\end{minipage}
\end{myalgo}}

\ifnotes
\usepackage{color}
\newcommand{\notename}[2]{{\textcolor{red}{\footnotesize{\bf (#1:} {#2}{\bf ) }}}}

\newcommand{\gnote}[1]{{\notename{Giacomo}{#1}}}
\newcommand{\lnote}[1]{{\notename{Laci}{#1}}}
\newcommand{\dnote}[1]{{\notename{Daniel}{#1}}}

\else

\newcommand{\notename}[2]{{}}

\newcommand{\gnote}[1]{}
\newcommand{\lnote}[1]{}
\newcommand{\dnote}[1]{}

\fi

\begin{document}

\title{Rescaling Algorithms for Linear Conic Feasibility\thanks{Partly based on the
  paper ``Rescaled coordinate descent methods for Linear Programming''
  presented at the  18th Conference on Integer Programming
and Combinatorial Optimization (IPCO 2016) \cite{IPCO-version}.}}
\author{
Daniel Dadush\thanks{Centrum Wiskunde \& Informatica. Supported by NWO Veni grant 639.071.510.}\\
\texttt{dadush@cwi.nl}
\and
L\'aszl\'o A. V\'egh\thanks{London School of Economics.} \thanks{Supported by EPSRC First Grant EP/M02797X/1.} \\
\texttt{l.vegh@lse.ac.uk}
\and
Giacomo Zambelli~\footnotemark[3] \\
\texttt{g.zambelli@lse.ac.uk}
}

\date{}

\maketitle

\abstract{We propose simple polynomial-time algorithms for two linear conic
  feasibility problems.
For a matrix $A\in \R^{m\times n}$, the {\em kernel problem} requires a positive vector in the
kernel of $A$, and the {\em image problem} requires a positive vector in the
image of $A^\T$.
 Both algorithms iterate between simple first order steps and rescaling steps. These rescalings improve natural geometric
potentials.
If Goffin's condition measure $\rho_A$ is
negative, then the kernel problem is feasible and the worst-case complexity of the kernel algorithm is
$O\left((m^3n+mn^2)\log{|\rho_A|^{-1}}\right)$; if $\rho_A>0$, then the image problem is feasible and
the image
algorithm runs in time
$O\left(m^2n^2\log{\rho_A^{-1}}\right)$. We also extend the image
algorithm to the oracle setting.

We address the degenerate case $\rho_A=0$ by extending our algorithms to find
maximum support nonnegative vectors in the kernel of $A$ and in the
image of $A^\top$. In this case the running time bounds are expressed in the bit-size model of computation: for an input matrix $A$ with integer entries and total encoding
length $L$, the maximum support kernel algorithm runs in time $O\left((m^3n+mn^2)L\right)$, while the  maximum support image
algorithm runs in time $O\left(m^2n^2L\right)$. The standard linear
programming feasibility problem can be easily reduced to either maximum
support problems, yielding polynomial-time algorithms for Linear Programming.}

\section{Introduction}\label{sec:intro}
We consider two fundamental linear conic
feasibility problems for an $m\times n$ matrix $A$. In the {\em kernel problem} the
goal is to find a positive vector in $\ker(A)$, whereas in the
{\em image problem} the goal is to find a positive vector in
$\im(A^\T)$. These can be formulated by the following feasibility problems.
 \noindent\begin{minipage}{.5\linewidth}
\begin{equation}\label{eq:main problem}\tag{$K_{++}$}
\begin{aligned}%
Ax&=0\\ x&>0
\end{aligned}
\end{equation}
\end{minipage}%
\noindent\begin{minipage}{.5\linewidth}
\begin{equation}\tag{$I_{++}$}\label{eq:main dual}
\begin{aligned}%
A^\T y>0
\end{aligned}
\end{equation}
\end{minipage}

We present simple polynomial-time algorithms for
the kernel problem  \eqref{eq:main problem}  and the image problem
\eqref{eq:main dual}. Both algorithms combine a first
order method with a geometric rescaling, which improve natural volumetric potentials.

The algorithms we propose fit into a line of research developed over the past 15
years
\cite{Basu-DeLoera-Junod,belloni,Betke,Chubanov-binary,Chubanov-s-poly,Chubanov-new,Dunagan-Vempala,li2015,penanew,Pena-Soheili,Roos,Vegh-Zambelli}.
These are polynomial algorithms for Linear Programming that combine simple
iterative updates, such as variants of perceptron \cite{Rosenblatt} or of the
relaxation method~\cite{Agmon,Motzin-Shoenberg}, with some form of geometric
rescaling.

\medskip

 Problems
\eqref{eq:main problem}  and \eqref{eq:main dual} have the following
natural geometric interpretations.
Let $a_1,\ldots,a_n$ denote the columns of the matrix $A$.
A feasible
solution to the \eqref{eq:main problem} means that $0$
is in the relative interior of the convex hull of the columns $a_i$, whereas a
feasible solution to
\eqref{eq:main dual} gives a hyperplane that strictly separates 0 from
the convex hull. We measure the running time of algorithms for
\eqref{eq:main problem} and \eqref{eq:main dual} in terms of Goffin's condition measure $\rho_A$,
where $|\rho_A|$ is the distance of 0 from the relative boundary of the convex hull of the vectors
$a_i/\|a_i\|$, $i\in[n]$. If
$\rho_A<0$, then \eqref{eq:main problem}  is feasible, if $\rho_A>0$, then
\eqref{eq:main dual} is feasible.

In case $\rho_A=0$, $0$ falls on the relative boundary of the convex hull of the
$a_i$'s, and  both problems  \eqref{eq:main problem}  and
\eqref{eq:main dual} are infeasible. We extend our kernel and image algorithms to
finding {\em maximum
support nonnegative vectors} in $\ker(A)$  and in $\im(A^\T)$, respectively. Geometrically, these amount to identifying
the face of the convex hull that contains $0$ in its relative interior. By
strong duality, the two maximum supports are complementary to each other.
The maximum support problems provides
fine-grained structural information on LP, and are crucial for exact LP
algorithms (see e.g. \cite{vavasis1995}). With either the
maximum support kernel or maximum support image algorithm, one can
solve
a general LP feasibility problem of the form $Ax\le b$ via
simple homogenization.
While LP feasibility (and thus LP optimization) can also
be reduced either to \eqref{eq:main problem} or to \eqref{eq:main dual} via
standard perturbation methods (see for example~\cite{Schrijver}), this is not
desirable for numerical stability.  Furthermore, we recall that the reduction from
LP optimization to feasibility creates degenerate (i.e.~non full-dimensional) systems by construction, and
hence in this sense most ``interesting'' LP feasibility problems are
degenerate.

\paragraph{Previous work}
We give a brief overview of geometric rescaling algorithms that
combine first order iterations and rescalings.
The first such algorithms were given by Betke
\cite{Betke} and by Dunagan and Vempala \cite{Dunagan-Vempala}. Both
papers address the problem \eqref{eq:main dual}. The deterministic algorithm of
Betke \cite{Betke} combines a variant of Wolfe's algorithm with a rank-one
update to the matrix $A$. Progress is measured by showing that the
spherical volume of the cone $A^\top y\ge 0$ increases by a constant factor at each rescaling. This approach
was further improved by Soheili and Pe\~na
\cite{Pena-Soheili} using different first order methods, in
particular, a smoothed perceptron algorithm
\cite{nesterov05,Pena-Soheili-smooth}.
Dunagan and Vempala \cite{Dunagan-Vempala} give a randomized
algorithm, combining two different first order methods. They also use
a rank-one update, but a different one from \cite{Betke}, and can show
progress directly in terms of Goffin's condition measure $\rho_A$.
The problem \eqref{eq:main dual} has been also studied in the more
general oracle setting \cite{belloni,Chubanov-oracle,penanew,Pena-Soheili}.

For \eqref{eq:main problem}, as well as for the maximum support version, a rescaling algorithm was given by
Chubanov \cite{Chubanov-new}, see also \cite{li2015,Roos}. A main iteration
of the algorithm concludes that in the system $Ax=0, 0\le x\le
\vec{e}$, one can identify at least one index $i$ such that $x_i\le
\frac12$ must hold for every solution. Hence
the rescaling multiplies $A$ from the right hand side by a diagonal
matrix.
(This is in contrast to the above mentioned algorithms,
where rescaling multiplies the matrix $A$ from the left hand side.)
The first order iterations are von Neumann steps on the system defined
by the projection matrix.

The algorithm \cite{Chubanov-new} builds on previous work by Chubanov
on binary integer programs and linear feasibility
\cite{Chubanov-binary,Chubanov-old-LP}, see also \cite{Basu-DeLoera-Junod}. A more
efficient variant of this algorithm was given in
\cite{Vegh-Zambelli}. These algorithms use a similar rescaling, but
for a non-homogeneous linear system, and the first order iterations
are variants of the relaxation method.

\paragraph{Our contributions}
We introduce new algorithms for \eqref{eq:main problem} and
\eqref{eq:main dual} as well as for their maximum support
versions, and improve on the running time bounds of the previous best geometric  rescaling
algorithm running time bounds in each of the settings.

For the {\em kernel problem}, that is, if $\rho_A<0$, we present a
simple new algorithm whose running time analysis is based on a new
volumetric potential, that serves as a proxy for $|\rho_A|$. In
contrast, \cite{Chubanov-new}  in essence reduces the problem to the
image setting. Using a direct volumetric argument enables us to obtain
a better running time in $O\left((m^3n+mn^2)\log{|\rho_A|^{-1}}\right)$ arithmetic operations.

For the {\em image problem}, that is, if  $\rho_A>0$, our new
algorithm is an enhanced version of Betke's
\cite{Betke} and Pe\~na and Soheili's \cite{Pena-Soheili}
algorithms. In contrast to rank-1 updates used in these
papers, we use higher rank updates.  The running
time of our algorithm is
$O\left(m^2n^2\log{\rho_A^{-1}}\right)$. This can be
improved to $O\left(m^3n\sqrt{\log n}\cdot \log{\rho_A^{-1}}\right)$
using smoothing techniques \cite{nesterov05,Pena-Soheili-smooth,Yu-Karzan-Carbonell}.
We also present an
extension of our algorithm for the case when the interior of a cone $\Sigma$
expressed by a separation oracle; the algorithm requires
$O\left(m^3\log{\rho_\Sigma^{-1}}\right)$ oracle calls and
$O\left(m^5\log{\rho_\Sigma^{-1}}\right)$  arithmetic operations,
where $\rho_\Sigma$ is the cone width. This oracle variant was
used in~\cite{submod} to develop new polynomial and strongly polynomial
algorithms for submodular function minimization.

We can obtain algorithms for the {\em maximum support versions} in both
settings by repeatedly applying the full support algorithm, and
observing the rescaled length of the column vectors. We show that if a
column vector becomes too long in the kernel setting (or too short in
the image setting) after a number of
rescalings, then it cannot be
contained in the maximum support. Thus, we can remove such vectors and
restart the algorithm.
For the maximum support image
problem, we obtain the  first rescaling algorithm.

These algorithms offer a particularly simple approach for
degenerate problems, even though these typically require substantial additional effort compared to the full
dimensional setting. For example, in the
ellipsoid method, the simultaneous Diophantine approximation technique
is used~\cite{glsbook}.
For interior point methods,
degeneracy must be dealt with both in the initialization phase, to set up a full
dimensional auxilliary system, and in the termination phase, to round to the optimal
face.

The full support algorithms can be implemented in
the {\em real model of computation}
\cite{blum1989} and the algorithms do not require an estimation of the
value of $|\rho_A|$. In contrast, for the maximum support algorithms, we need bounds
on the bit-complexity of the input to determine the threshold for
removing columns. Thus, we assume that
the input matrix $A$ is integer of total  encoding length
$L$. In this setting,  $|\rho_A|\ge 2^{-O(L)}$ whenever $\rho_A\neq 0$. This provides running time bounds
$O\left((m^3n+mn^2)\cdot L\right)$ for the full support kernel algorithm, and
$O\left(m^3n\sqrt{\log n}\cdot L\right)$ for the full support image
algorithm in the bit-complexity model. For the maximum support variants, the above described
column elimination method requires $n$ calls to the full support
algorithm in the kernel setting and $m$ calls in the image
setting. In the Appendix, we present improved versions for the maximum support variants of both
the kernel and image problems that can be implemented in the same
asymptotic complexity as the respective full support variants.
A summary of running times is given in
Table~\ref{table:running-time}.

\begin{table}[htb]
  \centering
  \begin{tabular}{|l|l|}
\hline
    \multicolumn{2}{|c|}{Kernel problem} \\
\hline
 Full support &   Maximum support  \\
\hline
$O(n^{18+3\varepsilon} \cdot L^{12+2\varepsilon} )$
    \cite{Chubanov-old-LP,Basu-DeLoera-Junod} &   \\
$O([n^5/\log n]\cdot L)$ \cite{Vegh-Zambelli} &  \\
   $O(n^4 \cdot L)$ \cite{Chubanov-new} & $O(n^4 \cdot L)$ \cite{Chubanov-new} \\
$O\left((m^3n+mn^2) \cdot \log |\rho_A|^{-1}\right)$  this paper &
                                                                 $O((m^3n+mn^2)    \cdot L) $
                                                                 this
                                                                 paper\\
\hline
\multicolumn{2}{c}{} \\
\hline
\multicolumn{2}{|c|}{Image problem} \\
\hline
 Full support &   Maximum support \\
\hline
$O(m^3n^3 \cdot \log \rho_A^{-1})$ \cite{Betke}& \\
$O(m^4 n\log m \cdot \log \rho_A^{-1})$\cite{Dunagan-Vempala}& \\
$O\left(m^{2.5}n^2\sqrt{\log n}\cdot \log \rho_A^{-1}\right)$ \cite{Pena-Soheili} & \\
$O\left(m^3n\sqrt{\log n}\cdot \log \rho_A^{-1}\right)$  this paper &  $O\left(m^3n\sqrt{\log
                                                n}\cdot L\right)$
                                                this paper\\
\hline
Full support oracle model& \\
\hline
$O(({\rm SO} \cdot m^5 + m^6) \rho_\Sigma^{-1})$ \cite{Pena-Soheili} & \\
$O(({\rm SO} \cdot m^4 + m^6) \rho_\Sigma^{-1})$ \cite{Chubanov-oracle} & \\
$O(({\rm SO} \cdot m^3 + m^5) \rho_\Sigma^{-1})$ this paper & \\
\hline
  \end{tabular}
  \caption{Running time of geometric rescaling algorithms. In the
    oracle setting,  $\rm SO$ is the complexity of a separation oracle call.}
  \label{table:running-time}
\end{table}

The full support kernel algorithm  was first presented in the conference version
\cite{IPCO-version}. The image algorithm and the maximum support variants for
both the kernel and dual problems are new in this paper.  The full support image
algorithm was also independently obtained by Hoberg and Rothvo{\ss}
\cite{rothvoss}.

The rest of the paper is structured as
follows. Section~\ref{sec:prelim} introduces notation and important
concepts.
Section~\ref{sec:coord-desc-overview} briefly surveys relevant first
order methods. In Sections~\ref{sec:kernel-full} and~\ref{sec:image-full} we give the algorithms for
the full support problems \eqref{eq:main problem} and \eqref{eq:main dual}, respectively. These are extended in Section~\ref{sec:max support} to
the maximum support cases. Variants with improved running times are
given in Appendix~\ref{sec:app-max}.


\subsection{Notation and preliminaries}\label{sec:prelim}

For a natural number $n$, let $[n]=\{1,2,\ldots,n\}$.
For a subset $X\subseteq [n]$, we let
$A_X\in \R^{m\times|X|}$ denote the submatrix formed by the columns of $A$
indexed by $X$. For any non-zero vector $v\in \R^m$ we denote by $\hat v$ the normal vector in the
direction of $v$, that is,
\[\hat v\eqdef \frac{v}{\|v\|}.\]
By convention, we also define $\hat{0} = 0$.
We let $\hat A\eqdef [\hat
a_1,\ldots,\hat a_n]$. Note that, given $v,w\in\R^m\sm\{0\}$, $\hat v^\T   \hat w$ is
the cosine of the angle between them. Given a vector $x\in\R^n$, the {\em support of $x$} is the subset of $[n]$ defined by $\supp(x)\eqdef \{i\in[n]\st x_i\neq 0\}$.

\medskip

Let $\R^n_+=\{x\in \R^n: x\ge 0\}$ and $\R^n_{++}=\{x\in \R^n: x> 0\}$
denote the set of nonnegative and  positive vectors in $\R^n$, respectively.
For any set $H\subseteq \R^n$, we let $H_+\eqdef H\cap\R^n_+$ and
$H_{++}\eqdef  H\cap \R^n_{++}$. These notations will be used in particular for the
kernel and image spaces
\[
\ker(A)\eqdef\{x\in\R^n\st Ax=0\},\quad
\im(A^\T)\eqdef\{A^\T y\st y\in \R^m\}\text{.}
\]
Clearly, $\im(A^\T)=\ker(A)^\bot$. Using this notation, \eqref{eq:main problem} is the problem of finding
a point in $\ker(A)_{++}$, and  \eqref{eq:main dual} amounts to finding
a point in $\im(A^\T)_{++}$. By strong duality, \eqref{eq:main problem}  is feasible if
and only if $\im(A^\T)_+=\{0\}$, that is,
\begin{equation}\label{eq:dual}
A^\T  y\geq 0,\tag{$I$}
\end{equation}
has no solution other than $y\in \ker(A^\T)$. Similarly, \eqref{eq:main dual} is
feasible if and only if $\ker(A)_+=\{0\}$, that is,
\begin{equation}\label{eq:primal}
\tag{$K$}
\begin{aligned}%
Ax&=0\\ x&\ge0
\end{aligned}
\end{equation}
has no solution other than $x=0$. Let us define

\[\Sigma_A\eqdef\{y\in \R^m: A^\T y\ge 0\}\] as the set of solutions to
\eqref{eq:dual}.

\medskip

Let  $I_d$ denote the $d$ dimensional identity matrix.
Let $\vec{e}_j$ denote the $j$th unit
vector, and $\vec{e}$ denote the all-ones vector of appropriate dimension
(depending on the context).
We denote by $\B^d$ the unit ball centered at the origin in $\R^d$, and let
$\nu_d=\vol(\B^d)$.

Given any set $C$ contained in $\R^d$, we denote by $\spn(C)$ the
linear subspace of $\R^d$ spanned by the elements of $C$. If $C\subseteq \R^d$ has affine dimension $r$, we denote by $\vol_r(C)$ the $r$-dimensional volume of $C$.

\paragraph{Projection matrices} For any matrix $A\in\R^{m\times n}$, we denote by $\Pi^I_A$ the {\em
  orthogonal projection matrix} to $\im(A^\T)$, and $\Pi^K_A$ as the
orthogonal projection matrix to $\ker(A)$. We recall that
\[\Pi^I_A=A^\T(AA^\T)^{+} A,\quad \Pi^K_A=I_n-A^\T(AA^\T)^{+} A,\]
 where $(\cdot)^+$ denotes the Moore-Penrose
pseudo-inverse. Note that, in order to compute $\Pi_A^I$ and $\Pi_A^K$, one does not
need to compute the pseudo-inverse of $AA^\T$; instead, if we let $B$
be a matrix comprised by $\rk(A)$ many linearly independent rows of
$A$, then $\Pi^I_A=B^\T(BB^\T)^{-1}B$, which can be computed in
$O(n^2m)$ arithmetic operations.

\paragraph{Scalar products}
We will often need to use scalar products and norms other than the
Euclidean ones. Let $\bb{S}^d_+$ and $\bb{S}^d_{++}$ be the sets of symmetric $d\times d$ positive semidefinite and positive definite matrices, respectively. Given   $Q\in \bb{S}^d_{++}$, we  denote by $Q^{\frac{1}{2}}$ the {\em square root} of $Q$, that is, the unique matrix in $\bb{S}^d_{++}$ such that $Q=Q^{\frac{1}{2}}Q^{\frac{1}{2}}$, and by $Q^{-\frac{1}{2}}$ the inverse of $Q^{\frac{1}{2}}$.
For $Q\in \bb{S}^d_{++}$ and two vectors $v,w\in \R^d$, we let
\[
\scalar{v,w}_Q \eqdef v^\top Q w,\quad \|v\|_Q \eqdef\sqrt{\scalar{v,v}_Q}.
\]
These define a scalar product and a norm over $\R^d$. We use
$\|\cdot\|_1$ for the $L^1$-norm and $\|\cdot\|_2$ for the Euclidean norm.
 When there is no risk of confusion we will simply
write $\|\cdot\|$ for $\|\cdot\|_2$.
Further, for any $Q\in \bb{S}^d_{++}$, we define the ellipsoid
\[
E(Q) \eqdef\{z\in \R^d\st \|z\|_Q\le 1\},
\]
and we recall that $E(Q)=Q^{-\frac{1}{2}}\bb{B}^d$ and $\vol(E(Q))=\det(Q)^{-\frac 12}\nu_d$.

We will use the following simple properties of matrix traces.
\begin{lemma}\label{lem:linalg} For any $X\in \bb{S}^{d}_+$,
\begin{enumerate}[(i)]
\item $\det(I_d+X)\ge 1+\tr(X)$.\label{det-trace-1}
\item   $\det(X)^{1/m}\le \tr(X)/m$.\label{det-trace-2}
\end{enumerate}
\end{lemma}
\begin{proof}Let $\lambda_1\ge\lambda_2\ge\ldots\ge\lambda_d\ge 0$
  denote the eigenvalues of $X$. {\bf (i)}  $\det(I_d+X)=\prod_{i=1}^d (1+\lambda_i)\ge1+ \sum_{i=1}^d \lambda_i=1+\tr(X)$,
where the equality is by the nonnegativity of the $\lambda_i$'s.  {\bf (ii)} By the inequality of arithmetic and geometric means, $\det(X)^{1/m}= (\prod_{i=1}^d \lambda_i)^{1/m}\le\sum_{i=1}^d \lambda_i/m=\tr(X)/m$.
\qued\end{proof}

\paragraph{The Goffin measure}
The running time of our full support algorithms will depend on the
following condition measure  introduced by Goffin \cite{Goffin}.
Given $A\in\R^{m\times n}$, we define
\begin{equation}\label{def:rho}
\rho_A\eqdef\max_{y\in\im(A)\setminus\{0\}} \min_{j\in [n]}\hat a_j^\T   \hat y.
\end{equation}
We remark that, in the literature, $A$ is typically assumed to have full row-rank
(i.e. $\rk(A)=m$), in which case $y$ in the above definition ranges
over all of $\R^m$. However, in some parts of the paper it will be
convenient not to make such an assumption.
The following Lemma summarizes well-known properties of $\rho_A$; the
proof will be given in Appendix~\ref{app:proofs} for completeness. For the matrix
$A$, we let $\conv(A)$ denote the convex hull of the column vectors of
$A$.

\begin{restatable}{lemma}{rhodef}\label{lem:rho}
$|\rho_A|$ equals the distance of 0 from the relative boundary of
$\conv(\hat A)$. Further,
\begin{enumerate}[(i)]
\item  $\rho_A<0$ if and only if 0 is in the relative  interior of $\conv(A)$.
  \item $\rho_A>0$ if and only if 0 is outside $\conv(A)$. In this case, the Goffin measure $\rho_{A}$ equals the
   {\em width of the image cone} $\Sigma_A$, that is, the radius of the
largest ball in $\R^m$ centered on the surface of the unit sphere and inscribed in
$\Sigma_A$.
\end{enumerate}
\end{restatable}

The following quantities will be needed for bit complexity estimations. Assume that
the $m\times n$ matrix $A$ has integer entries and encoding size  $L$. Letting $\mathcal{B}=\{B\subseteq [n]\st |B|=\rk(A_B)\}$, we
define
\begin{equation}\label{def:delta}
\Delta_A=\max_{B\in\mathcal{B}} \prod_{j\in B}\|a_j\|,\quad\mbox{and}\quad \bart_A=\frac {1}{m^2{\Delta_A}^2}.
\end{equation}

\begin{restatable}{lemma}{numerical}\label{lem:numerical-shit}
Let $A\in \Z^{m\times n}$ of total encoding length $L$. If $\rho_A\neq 0$, then  $\displaystyle|\rho_A|\geq \bart_A \geq 2^{-4L}$.
\end{restatable}
The proof can be found in Appendix~\ref{app:proofs}. Let us note that $\Delta_A$ and $\bart_A$ can be efficiently
computed. Indeed, the value of $\Delta_A$ is the maximum weight base of a
linear matroid, and can be computed by the greedy algorithm in $O(m^2
n+n\log n)$ arithmetic operations,
since this amounts to sorting the columns of $A$ according to their length, and then applying Gaussian elimination (which requires $O(m^2n)$ operations).

\subsection{First order algorithms}\label{sec:coord-desc-overview}

Various first order methods are known for finding non-zero
solutions to \eqref{eq:primal} or to  \eqref{eq:dual}. Most algorithms
assume either the feasibility of \eqref{eq:main problem}, or the feasibility of \eqref{eq:main
  dual}. We outline the
common update steps of these algorithms.

At every iteration, maintain a non-negative, non-zero vector $x\in
\R^n$, and we
let $y=Ax$. If $y=0$, then $x$ is a non-zero point in $\ker(A)_+$. If  $A^\T  y> 0$,
then $A^\T y\in \im(A)_{++}$. Otherwise, choose an index $k\in [n]$ such that
$a_k^\T  y \le 0$, and update $x$ and $y$ as follows:
\begin{equation}\label{eq:general update} y':=\alpha y+\beta \hat a_k;\quad
x':=\alpha x+\frac{\beta}{\|a_k\|} \vec{e}_k, \end{equation} where
$\alpha,\beta>0$ depend on the specific algorithm. Below we discuss various
possible update choices.

\paragraph{Von Neumann's algorithm} The vector $y$ is maintained
throughout as a convex combination of $\hat a_1,\ldots,\hat a_n$.
The parameters $\alpha, \beta>0$ are chosen so that $\alpha+\beta=1$ and
$\|y'\|$ is smallest possible. That is, $y'$ is the point of minimum norm on the
line segment joining $y$ and $\hat a_k$. If we denote by $y^t$ the vector at
iteration $t$, and initialize $y^1=\hat a_k$ for an arbitrary $k\in [n]$, a simple argument shows that $\|y^t\|\leq 1/\sqrt{t}$ (see
Dantzig \cite{Dantzig-92}).  If 0 is contained in the interior of the convex
hull, that is $\rho_{A}<0$, Epelman and Freund \cite{Epelman-Freund} showed
that $\|y^t\|$ decreases by a factor of $\sqrt{1-\rho_{A}^2}$ in every
iteration. 

If $0$ is outside the convex hull  that is, $\rho_{A}>0$, then the algorithm terminates
after at most $1/\rho_{A}^2$ iterations.  A recent result by Pe\~na, Soheili, and
Rodriguez \cite{Pena-Soheili-away} gives a variant of the algorithm with a
provable convergence guarantee in the case $\rho_{A}=0$, that is, if $0$ is
on the boundary of the convex hull.

Among the previous geometric rescaling algorithms, variants of von
Neumann's algorithm have been used by  Betke~\cite{Betke} for the case
$\rho_A>0$, and by Chubanov~\cite{Chubanov-new} for the case
$\rho_A<0$. We will use this method in our Image Algorithm, that is,
for $\rho_A>0$.

We note that von Neumann's algorithm is the same as the Frank-Wolfe
conditional gradient descent method \cite{frankwolfe}  with optimal step size for the quadratic program $\min \|Ax\|^2$ s.t. $\vec{e}^\T x=1, x\ge 0$.

\paragraph{Perceptron algorithm}  The Perceptron algorithm chooses $\alpha=\beta=1$ at every
iteration. If $\rho_A>0$, then, similarly to the von Neumann algorithm, the Perceptron algorithm terminates
with a solution to the system \eqref{eq:main dual}  after at most $1/\rho_A^2$
iterations (see Novikoff \cite{Novikoff}).
The Perceptron and von Neumann algorithms can be interpreted as duals of each other, see Li and Terlaky \cite{Terlaky-Li}.

Pe\~na and Soheili gave a smoothed
variant of the Perceptron update which guarantees termination in time
$O(\sqrt{\log n}/|\rho_A|)$ iterations~\cite{Pena-Soheili-smooth}. This
was used in the polynomial-time rescaling algorithm
\cite{Pena-Soheili} for $\rho_A>0$.
The same iteration bound
$O(\sqrt{\log n}/|\rho_A|)$ was achieved by Yu et
al.~\cite{Yu-Karzan-Carbonell} by adapting the Mirror-Prox algorithm
of Nemirovski~\cite{Nemirovski}.

With a normalization to $\vec{e}^\T x=1$, Perceptron implements the
Frank-Wolfe algorithm for the same system
 $\min \|Ax\|^2$ s.t. $\vec{e}^\T x=1, x\ge 0$, with step length
 $1/(k+1)$ at iteration $k$.
An alternative view is to interpret Perceptron as a coordinate descent method for the
system $\min \|Ax\|^2$ s.t. $x\ge \vec e$, with
fixed step length 1 at every iteration.

\paragraph{Dunagan-Vempala algorithm} The first order method
used in \cite{Dunagan-Vempala}
chooses $\alpha=1$ and $\beta=-(\hat a^\T  _k y)$. The  choice of
$\beta$ is the one that makes $\|y'\|$ the smallest possible when
$\alpha=1$.
 It follows that $\|y'\|^2=\|y\|^2+2\beta(\hat a_k^\T   y)+\beta^2\|\hat a_k\|^2=\|y\|^2-2(\hat a_k^\T   y)^2+(\hat a_k^\T  y)^2=\|y\|^2(1-(\hat a_k^\T  \hat y)^2)$, hence
\begin{equation}\label{eq:decrease in norm}
\|y'\|= \|y\|\sqrt{1-(\hat a_k^\T  \hat y)^2}.
\end{equation}
In particular, the norm of $y'$ decreases at every iteration, and the
larger is the angle between $a_k$ and $y$, the larger the decrease. If
$\rho_A<0$, then $|\hat a_k^\T  \hat y|\geq |\rho_A|$, therefore this guarantees a decrease in the norm of
at least $\sqrt{1-\rho_A^2}$.

This is a coordinate descent for the system $\min \|Ax\|^2$ s.t. $x\ge \vec e$, but with the optimal step length.
One can also interpret it as the Frank-Wolfe algorithm with the
optimal step length for the same system.\footnote{The Frank-Wolfe method is originally
  described for a compact set, but the set here is
  unbounded. Nevertheless, one can easily modify the method by moving
  along an unbounded recession direction.}

While Dunagan and Vempala used these updates for $\rho_A>0$, we will
use them in our Kernel Algorithm for $\rho_A<0$.

\section{The Full Support Kernel Algorithm}\label{sec:kernel-full}

The Full Support Kernel Algorithm (Algorithm~\ref{alg:primal-algorithm-matrix}) solves the full support problem \eqref{eq:main
  problem}, that is, it finds a point in $\ker(A)_{++}$, assuming that
$\rho_A<0$, or equivalently,
$\ker(A)_{++}\neq \emptyset$. We assume that the columns of input
matrix $A$ are normalized, that is, $A=\hat A$. However, this property
does not hold throughout the algorithm, since
the rescalings will modify the length of the columns.
We use the parameter
\begin{equation}\label{eq:definition epsilon}\varepsilon\eqdef \const.
\end{equation}

\renewcommand{\algorithmicrequire}{\textbf{Input:}}
\renewcommand{\algorithmicensure}{\textbf{Output:}}

\begin{figure}[htb]
\begin{center}
\begin{minipage}{0.85\textwidth}
\begin{algorithm}[H]
\raggedright
  \begin{algorithmic}[1]
    \Require{A matrix $A\in\R^{m\times n}$  such that $\rho_A<0$ and $\|a_j\|=1$ for all $j\in [n]$.}
    \Ensure{A solution to the system (\ref{eq:main problem}).}
    \State Compute $\Pi:=\Pi_A^K=I_n-A^\T( AA^\T)^{+} A$.
    \State Set $x_j:=1$ for all $j\in [n]$,  and $y:=Ax$.
    \While{$\Pi x\not> 0$}
  	\State Let $\displaystyle k:=\arg\min_{j\in [n]} \hat a_j^\T   \hat y$;
       \If{$\hat a_k^\T   \hat y< -\varepsilon$}
          \State {\bf update} $\displaystyle x:=x-\frac{a_k^\T
            y}{\|a_k\|^2}\vec{e}_k$;\quad $\displaystyle y:=y-(\hat
          a_k^\T   y) \hat a_k$; \label{li:DV}
         \Else
\State {\bf rescale}
 $\displaystyle A:=\left(I_m+\hat y\hat y^\T  \right)A$;\quad  $\displaystyle y:=2y$;\label{li:Kernel rescale}
	\EndIf
    \EndWhile
    \Return{$\bar  x=\Pi x$ as a feasible solution to (\ref{eq:main problem}).}
 \end{algorithmic}
\caption{Full Support Kernel Algorithm}\label{alg:primal-algorithm-matrix}
\end{algorithm}
\end{minipage}
\end{center}
\end{figure}

 We use Dunagan-Vempala (DV) updates as the
first order method. We maintain a vector $x\in \R^n$, initialized as $x=\vec{e}$; the coordinates $x_i$ never decrease during the algorithm.
We also maintain $y=Ax$, and a main quantity of interest is the norm $\|y\|^2$.
In each iteration of the algorithm, we check whether $\bar x=\Pi x$, the projection of $x$ onto $\ker(A)$, is strictly positive.
If this happens, then
$\bar x$ is returned as a feasible solution to \eqref{eq:main
  problem}.

Every iteration performs either  a DV update to $x$ (line~\ref{li:DV}) or a rescaling of the matrix $A$ (line~\ref{li:Kernel rescale}).
Because DV updates are performed only for $k\in [n]$ satisfying $\hat a_k^\T   \hat y< -\varepsilon$, \eqref{eq:decrease in norm} ensures a substantial decrease in the norm, namely
\begin{equation}\label{eq:decrease in norm epsilon}
\|y'\|\leq \|y\|\sqrt{1-\varepsilon^2}.
\end{equation}

On the other hand, rescalings are performed when if $\hat a_j^\T   \hat y\geq -\varepsilon$ for all
$j\in[n]$, which implies that $|\rho_A|<\eps $. In this case, we show a substantial improvement in a volumetric potential.
Let us define the polytope $P_A$ as
\begin{equation}\label{eq:P_A}P_A\eqdef \conv(\hat A)\cap (-\conv(\hat
A)).\end{equation}
The volume of $P_A$ will be used as a proxy to
$|\rho_A|$.
Indeed, from Lemma~\ref{lem:rho}  $|\rho_A|$ is the radius of the
largest ball around the origin inscribed in $\conv(\hat A)$, and this ball
must be contained in $P_A$.

The condition $\hat a_j^\T   \hat y\geq -\varepsilon$ for all $j\in [n]$ implies then $P_A$  is contained in a ``narrow strip'' of width $2\varepsilon$, namely $P_A\subseteq \{z\in\R^m\st -\varepsilon\leq \hat y^\T   z\leq \varepsilon\}$. If we replace $A$ with the matrix $A':=(I+\hat y\hat
y^\T  )A$,  then Lemma~\ref{lem:volume-increase} implies that
$\vol(P_{A'})\geq \nicefrac{3}{2}\vol(P_A)$. Geometrically, $A'$ is obtained
by applying  to  the columns of $A$ the linear transformation that
``stretches'' them by a factor of two in the direction of $\hat y$
(see Figure~\ref{fig:rescale}).

\begin{figure}[htbp]
\begin{center}
\begin{minipage}{\ourpagewidth}
\psfrag{y}{$\hat y$}
\psfrag{1}{$a_1$}
\psfrag{2}[c]{$a_2$}
\psfrag{3}{$a_3$}
\psfrag{4}{$a_4$}
\psfrag{5}{$a_5$}
\psfrag{a}{$a'_1$}
\psfrag{b}{$a'_2$}
\psfrag{c}{$a'_3$}
\psfrag{d}{$a'_4$}
\psfrag{e}{$a'_5$}
\psfrag{P}{$P_A$}
\psfrag{Q}{$P_{A'}$}
\psfrag{v}{$\varepsilon$}
\includegraphics[scale=.5]{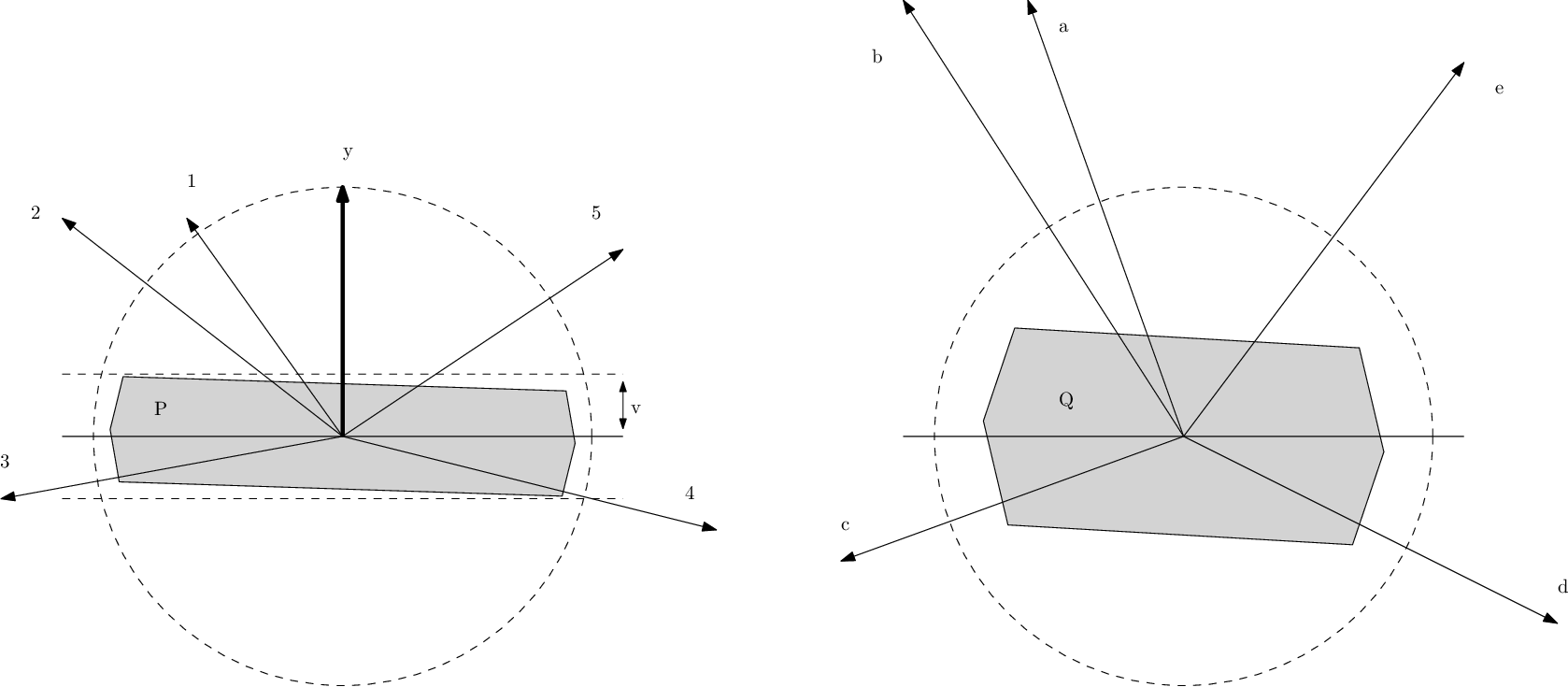}
\caption{\label{fig:rescale} \small{Effect of rescaling. The dashed circles represent the points of norm 1. The shaded areas are $P_A$ and $P_{A'}$.}}
\end{minipage}
\end{center}
\end{figure}

Thus, at every iteration we either have a substantial decrease in the length of
the current $y$, or we have a constant factor increase in the volume of $P_A$.

The volume of $P_A$ is bounded by the volume of the unit ball in
$\R^m$, and initially contains a ball of radius $|\rho_A|$ around the origin. Consequently, the number of rescalings cannot exceed
$ m\log_{\nicefrac{3}{2}}{|\rho_A|^{-1}} $.

The norm $\|y\|$ changes as follows. In every iteration where the DV update
is applied, the norm of $\|y\|$  decreases by a factor
$\sqrt{1-\varepsilon^2}$ according to \eqref{eq:decrease in norm
  epsilon}. At every rescaling, the norm of $\|y\|$ increases by a
factor 2.  Lemma~\ref{lem:delta-rounding} shows that once
$\|y\|<|\rho_A|$ for the  initial value of $|\rho_A|$, then the algorithm
terminates with $\bar x=\Pi x>0$. We will prove the following running
time bounds.

\begin{theorem}\label{thm:main-matrix} For any input matrix $A\in\R^{m\times
    n}$ such $\rho_A<0$ and $\|a_j\|= 1$ for all $j\in [n]$,  Algorithm~\ref{alg:primal-algorithm-matrix}
  finds a feasible solution of~\eqref{eq:main problem}  in
  $O(m^2\log n+m^3\log{|\rho_A|^{-1}})$ DV updates. The number of arithmetic
  operations is $O(m^2n\log n$ $+(m^3n+mn^2)\log{|\rho_A|^{-1}})$.
\end{theorem}

Using Lemma~\ref{lem:numerical-shit}, we obtain a running time bound in terms of
bit complexity.

\begin{corollary}\label{cor:bit complexity primal} Let $A$ be an $m\times n$ matrix with integer entries and encoding size $L$. If $\rho_A<0$, then Algorithm~\ref{alg:primal-algorithm-matrix} applied to $\hat A$ finds a feasible solution of~\eqref{eq:main problem} in $O\left((m^3n+mn^2)L\right)$ arithmetic operations.
\end{corollary}

\paragraph{Coordinate Descent with Finite Convergence}
Before proceeding to the proof of Theorem~\ref{thm:main-matrix}, let
us consider a modification of Algorithm~\ref{alg:primal-algorithm-matrix} without any
rescaling. That is, at every iteration we perform a DV update (even if
$\hat a_j^\T   \hat y\geq -\varepsilon$ for all $j\in [n]$), until
$\Pi_A^K x>0$.  We claim that if $\rho_A<0$, then the total number of DV steps is bounded by $O(\log(n/|\rho_{A}|)/\rho_{A}^2)$.

This is in contrast with von Neumann's algorithm that does not have
finite convergence for $\rho_A<0$; this aspect is discussed by Li and Terlaky
\cite{Terlaky-Li}. Dantzig \cite{dantzig1991} proposed a finitely
converging variant of
von Neumann's algorithm, but this involves running the algorithm $m+1$
times, and also an explicit lower bound on the parameter $|\rho_A|$. Our
algorithm does not incur a running time increase compared to the
original variant, and does not require such a bound.

Let us now verify the running time bound of our variant.
Again, let us assume $\|a_j\|=1$ for
all $j\in [n]$ for the input.
It follows by  \eqref{eq:decrease in norm
  epsilon} that the norm $\|y\|$ decreases by at least a factor
$\sqrt{1-\rho_A^2}$ in every DV update. Initially, $\|y\|\le {n}$,  and, as shown in
Lemma~\ref{lem:delta-rounding}, the algorithm terminates with a solution $\Pi_A^K x>0$ as soon as $\|y\|<|\rho_A|$.
This yields the bound $O(\log(n/|\rho_{A}|)/\rho_{A}^{2})$ on the number
of DV steps.

\subsection{Analysis}\label{sec:analysis rho}
We will use the following technical lemma.
\begin{lemma}
\label{lem:exp2}
Let $X \in \R$ be a random variable supported on the interval $[-\eps,\eta]$,
where $0 \leq \eps \leq \eta$, satisfying $\E[X] = \mu$. Then for any $c \geq 0$, we
have that $\E[\sqrt{1+ cX^2}] \leq \sqrt{1+c\eta(\eps+|\mu|)}$.
\end{lemma}
\begin{proof}
Let $l(x) = \frac{\eta-x}{\eta+\eps} \sqrt{1+c \eps^2} +
\frac{x+\eps}{\eta+\eps} \sqrt{1+c \eta^2}$ denote the unique affine
interpolation of $\sqrt{1+cx^2}$ through the points $\set{-\eps, \eta}$. By
convexity of $\sqrt{1+cx^2}$, we have that $l(x) \geq \sqrt{1+cx^2}$ for all $x \in
[-\eps,\eta]$. It follows that
$\E[\sqrt{1+cX^2}]\leq \E[l(X)]= l(\E[X]) = l(\mu)$, where the first
equality holds since $l$ is affine.  From here, we get that
\begin{align*}
l(\mu) &= \frac{\eta-\mu}{\eta+\eps} \sqrt{1+c \eps^2} +
          \frac{\mu+\eps}{\eta+\eps} \sqrt{1+c \eta^2} \\
			 &\leq \sqrt{1+c\left(\frac{\eta-\mu}{\eta+\eps} \eps^2 +
			       \frac{\mu+\eps}{\eta+\eps}\eta^2\right)}
          \quad \left(\text{by concavity of $\sqrt{x}$}\right) \\
			 &= \sqrt{1+c\left(\eta \eps + (\eta-\eps)\mu \right)}
			 \leq \sqrt{1+c \eta(\eps + |\mu|)}
          \quad \left(\text{since $\eps \leq \eta$}\right) \text{,}
\end{align*}
as needed.
\qued \end{proof}

The crucial part of the analysis is to bound the volume increase of $P_A$
at every rescaling iteration.

\begin{lemma}
\label{lem:volume-increase} Let $A\in\R^{m\times n}$ and let $r=\rk(A)$. For some $0 < \eps \le
1/(11 r)$, let  $v \in \R^m$, $\|v\|=1$, such that $\hat{a}_j^\T   v \geq
-\eps~\forall j \in [n]$. Let $T = (I+vv^\T  )$, and let $A' = TA$. Then
\begin{enumerate}[(i)]
\item $TP_{A}\subseteq (1+3\varepsilon)P_{A'}$,
\item If $v\in\im(A)$, then $\vol_r(P_{A'}) \geq \nicefrac32\vol_r(P_A)$.
\end{enumerate}
\end{lemma}
\begin{proof} {\bf (i)} The statement is trivial if $P_A=\emptyset$, thus we assume $P_A\neq\emptyset$. Consider an
arbitrary point $z\in P_A$. By symmetry, it suffices to show $Tz\in
(1+3\varepsilon)\conv(\hat A')$.
By definition, there exists $\lambda\in\R^n_+$ such that $\sum_{j=1}^n \lambda_j= 1$ and $z=\sum_{j=1}^n \lambda_j \hat a_j$.
Note that
\[
Tz=\sum_{j=1}^n \lambda_j T\hat a_j=\sum_{j=1}^n (\lambda_j \|T\hat a_j\|) \hat a'_j
=\sum_{j=1}^n \lambda_j \sqrt{1+3(v^\T   \hat a_j)^2}\,\hat a'_j\text{.} \]
Since $P_{A'}\neq\emptyset$, it follows that $0\in \conv(\hat A')$, thus it suffices to show that $\sum_{j=1}^n \lambda_j \sqrt{1+3(v^\T   \hat a_j)^2}\leq
1+3\varepsilon\text{.}$ The latter expression is of the form $\E[\sqrt{1+3X^2}]$ where $X$ is a random
variable supported on $[-\varepsilon, 1]$ and $|\E[X]| = |\sum_{j=1}^n \lambda_j
v^\T   \hat{a}_j| = |v^\T   z|$. Note that $|v^\T  z| \leq \varepsilon$ because both $z$
and $-z$ are in $P_A$. Hence, by Lemma~\ref{lem:exp2}, \[\sum_{j=1}^n \lambda_j \sqrt{1+3(v^\T   \hat a_j)^2}\le
\sqrt{1+3(2\varepsilon)}\le 1+3\varepsilon.\]

\noindent{\bf (ii)} Note that $\vol_r(TP_A)=2\vol_r(P_A)$ as $\det(T)=2$. Thus
we obtain $\vol_r(P_{A'})\ge 2\vol_r(P_A)/(1+3\varepsilon)^r\ge \nicefrac32\vol_r(P_A)$,
since $(1+3\varepsilon)^r\le (1+{3}/{(11m)})^r\leq e^{\nicefrac{3}{11}}\leq \nicefrac43$.
\qued
\end{proof}

\begin{lemma}
\label{lem:delta-rounding} Let $A\in\R^{m\times n}$ with $\|a_j\|=1$ for all $j\in[n]$. Given $x\in\R^n$ such that $x\geq\vec{e}$, if $\|A x\| <|\rho_{(A,-A)}|$ then $\Pi_A^K x>0$. In particular, if $\rho_A<0$, then $\Pi_A^K x>0$ whenever $\|A x\| <|\rho_{A}|$.
\end{lemma}
\begin{proof} Let $\Pi:= \Pi_A^K$ and define $\delta\eqdef \min_{j\in[n]} \|(AA^\T)^{+}a_j\|^{-1}$. Observe that, if $\|A x\| < \delta$, then $\Pi x>0$. Indeed, for $j \in [n]$, $(\Pi x)_j =   x_j - a_j^\T (AA^\T)^+  y \geq 1 - \|(AA^\T)^+ a_j\|\| Ax\|    > 1 - \delta^{-1}\delta = 0$, as required.

Thus it suffices to show that $\delta\geq \displaystyle |\rho_{(A,-A)}|$. Let $k:=\arg\max_{j\in [n]}\|(AA^\top)^+ a_j\|$, define $z:=(AA^\T)^+ a_k$, and note that $\|z\|=1/\delta$. Note that $\rho_{(A,-A)}<0$, thus
\[
|\rho_{(A,-A)}| = -\rho_{(A,-A)} = \min_{y\in\im(A)\smz} \max_{j \in [n]} |a_j^\T \hat y|\leq \max_{i \in [n]} |a_i^\T \hat z|=\delta\max_{j \in [n]} |{\Pi_{jk}}|
                \leq \delta
\text{.}
\]
where the last inequality follows from the fact that $|\Pi_{ij}|\leq 1$ for all $i,j\in [n]$. The last part of the statement follows from the fact that $|\rho_A|\leq |\rho_{(A,-A)}|$
\qued \end{proof}

\begin{lemma}\label{lemma:rescaling stretch} Let $v\in\R^m$, $\|v\|=1$. For any
$y,\bar y\in \R^m$ such that $y=(I+vv^\T)\bar y$, we have $\|\bar y\|\leq \|y\|$.
\end{lemma}
\begin{proof}
We have $\|y\|^2=\|\bar y+(v^\T \bar y)v\|^2=\|\bar y\|^2+2(v^\T \bar y)^2+(v^\T \bar y)^2\|v\|^2\geq \|\bar y\|^2$.
\qued \end{proof}

\begin{proof}[Proof of Theorem~\ref{thm:main-matrix}]
We use $\bar A$ for the input matrix and $A$ for the current matrix
during the algorithm, so that, after $k$ rescalings, $A=(I+v_1v_1^\T)\cdots
(I+v_kv_k^\T)\bar A$ for some vectors $v_1,\ldots,v_k\in \im(A)$ with
$\|v_i\|=1$ for $i\in [n]$.

Let $\rho:=|\rho_{\bar A}|$
and $\Pi:=\Pi^K_{\bar A}$. Note that $\ker(A)=\ker(\bar A)$, and hence
$\Pi^K_A=\Pi$ throughout the algorithm. The matrix $\Pi$ needs to be computed only once, requiring $O(n^2m)$ arithmetic operations, which is clearly dominated by the stated running-time of the algorithm.

Let $x$ and
$y=Ax$ be the vectors computed in every iteration, and define $\bar
y:=\bar Ax$ and $\bar x=\Pi x$. Lemma~\ref{lemma:rescaling stretch} implies that $\|\bar
y\|\leq \|y\|$, thus it follows from Lemma \ref{lem:delta-rounding} that $\bar x>0$ whenever $\|y\|< \rho$.  This shows that the algorithm terminates with the $\bar x$ to \eqref{eq:main problem}  as soon as $\|y\|<\rho$.

As previously discussed, Lemma~\ref{lem:volume-increase} implies that
the number $K$ of rescalings cannot exceed
$m\log_{\nicefrac{3}{2}}{|\rho|^{-1}}$. At every rescaling, $\|y\|$
increases by a factor 2. In every iteration where the DV update is
applied, $\|y\|$  decreases by a factor $\sqrt{1-\varepsilon^2}$
according to \eqref{eq:decrease in norm epsilon}. Initially, $y=\bar
A\vec{e}$, therefore  $\|y\|\le {n}$ since all columns of $A$ have
unit norm. This shows that the number of DV iterations is bounded by $\kappa$, where $\kappa$ is the smallest integer such that $n^2(1-\varepsilon^2)^\kappa 4^K<\rho^2$.
Taking the logarithm on both sides, and using the fact that
$\log(1-\varepsilon^2)<-\varepsilon^2$, it follows that $\kappa\in
O(m^2(\log n+K+\log|\rho|^{-1}))=O(m^2\log n+m^3\log{|\rho|^{-1}})$.

We can implement every DV update in $O(n)$ time, at the cost of an
$O(n^2)$ time preprocessing at every rescaling, as explained next.
After every rescaling we compute the matrix
$F:=A^\T A$  and the norms of the columns of $A$. Computing the norms requires time $O(nm)$. The matrix $F$ is updated as $F:=A^\T(I+\hat y\hat y^\T)^2
A=AA^\T+3(A^\T\hat y\hat y^\T A)=F+3 z z^\T/\|y\|^2$, which requires
time $O(n^2)$.

Further, at every DV update, we maintain the vectors $z=A^\T y$ and $\bar x=\Pi x$.
Using the vector $z$, we can compute  $\arg\min_{j\in [n]}
\hat a_j^\T   \hat y=\arg\min_{j\in [n]} z_j/\|a_j\| $ in time
$O(n)$ at any DV update. We also need to
recompute $y,z$, and $\bar x$. Using $F=[f_1,\ldots,f_n]$, these can be obtained as $y:=y-(\hat a_k^\T   y) \hat a_k$,
$z:=z-f_k(\hat a_k^\T   y)/\|a_k\|$, and $\bar x:=\bar x-
\Pi_{k}(\hat a_k^\T   y) /\|a_k\|$, where $\Pi_k$ denotes the $k$th column of $\Pi$. These updates altogether take $O(n)$
time.

Therefore the number of arithmetic operations is $O(n)$
times the number of DV updates plus $O(n^2)$ times the number of
rescalings. The overall running time estimate
follows.
\qued \end{proof}

\section{The Full Support Image Algorithm}\label{sec:image-full}
The Image Algorithm maintains a positive definite matrix $Q$,
initialized as $Q=I_m$. We use the von Neumann algorithm
(Algorithm~\ref{alg:Neumann}) as the first order method, with the
scalar product $\scalar{.,.}_Q$.
Within $O(m^2)$ iterations, the von Neumann algorithm obtains a vector
$y\in\conv(a_1/\|a_1\|_Q,\ldots,a_n/\|a_n\|_Q)$ with  $\|y\|_Q\le
\varepsilon$. Then, we update the matrix $Q$, using the coefficients of the convex
combination.

Algorithm~\ref{alg:Neumann} is same as von Neumann's algorithm as described by
Dantzig  \cite{Dantzig-92}, with the standard scalar product replaced by
$\scalar{.,.}_Q$  for a
 matrix $Q\in\bb{S}^m_{++}$, and using the normalized columns $a_i/\|a_i\|_Q$.
We remark that running the algorithm with $\scalar{.,.}_Q$ is the same as running it for the standard scalar product for the unit vectors $Q^{1/2}a_i/\|Q^{1/2}a_i\|_2$.

\renewcommand{\algorithmicrequire}{\textbf{Input:}}
\renewcommand{\algorithmicensure}{\textbf{Output:}}

\begin{figure}[htb!]
\begin{center}
\begin{minipage}{0.85\textwidth}
\begin{algorithm}[H]
\raggedright
  \begin{algorithmic}[1]
    \Require{A matrix $A\in\R^{m\times n}$, a positive definite
      matrix $Q\in \R^{m\times m}$ and an $\varepsilon>0$.}
    \Ensure{Vectors $x\in \R^n$, $y\in \R^m$  such that
      $y=\sum_{i=1}^n{x_i a_i/\|a_i\|_Q}$, $\vec e^\top x=1$, $x\ge 0$, and
      either  $A^\T Qy>0$ or  $\|y\|_Q\le \varepsilon$.}
    \State Set $x:=\vec e_1$, $y:={a_1}/{\|a_1\|_Q}$.
    \While{$\|y\|_Q>\varepsilon$}
        \If{$\scalar{a_i,y}_Q  >0$ for all $i\in[n]$} \Return{$x$ and $y$ satisfying $A^\T Qy>0$.  }
        \State{Terminate.}
    \Else{ Select $k\in[n]$ such that $\scalar{a_k,y}_Q \leq 0$;}
      \State Let $\lambda:=\displaystyle\frac{\scalar{y-a_k/\|a_k\|_Q,y}_Q}{\|y-a_k/\|a_k\|_Q\|^2_Q}$;
       \State {\bf update} $\displaystyle x:=(1-\lambda)x+\lambda
       \vec{e}_k$; \quad $\displaystyle y:=(1-\lambda)y+\lambda\frac{a_k}{\|a_k\|_Q}$;
   \EndIf
	\EndWhile
\State{\Return{the vectors $(x,y)$.}}
 \end{algorithmic}
\caption{The von Neumann algorithm}\label{alg:Neumann}
\end{algorithm}
\end{minipage}
\end{center}
\end{figure}

\begin{lemma}\label{lemma:VN running time}
For a given $\varepsilon>0$, the von Neumann algorithm terminates in at
most  $\lceil 1/\varepsilon^2\rceil$
updates. Each iteration requires $O(n)$ arithmetic operations, provided that the matrix $A^\T QA$ has been precomputed.
\end{lemma}
 \begin{proof}
The  $\lceil 1/\varepsilon^2\rceil$ bound on the number of iterations is due to Dantzig \cite{Dantzig-92}.
If we maintain the vector $z:=A^\T Q y$, then checking whether or not $\scalar{a_i,y}_Q>0$ for all $i\in[n]$ amounts to checking if $z>0$,
which can be performed in time $O(n)$. Recomputing $x':=(1-\lambda)x+\lambda\vec{e}_k$ and $y':=(1-\lambda)y+\lambda{a_k}/{\|a_k\|_Q}$ requires time $O(n)$.
Recomputing $z':=A^\T Qy$ requires to compute
$z':=(1-\lambda)z+\lambda A^\T Q a_k/\|a_k\|_Q$, which can also be done in time $O(n)$ provided that $A^\T Q A$ has been precomputed.
 \qued \end{proof}

Algorithm~\ref{alg:dual-alg-full} shows the Full Support Image
Algorithm. We set the same $\varepsilon=\const$ as in the kernel algorithm.
Without loss of generality, we can  assume that the matrix A has full row rank, that is,
$\im(A)=\R^m$.
\renewcommand{\algorithmicrequire}{\textbf{Input:}}
\renewcommand{\algorithmicensure}{\textbf{Output:}}
\begin{figure}[htb]
\begin{center}
\begin{minipage}{0.85\textwidth}
\begin{algorithm}[H]
\raggedright
  \begin{algorithmic}[1]
    \Require{A matrix $A\in\R^{m\times n}$ such that $\rk(A)=m$ and
      \eqref{eq:main dual} is feasible.}
    \Ensure{A feasible solution to  \eqref{eq:main dual}.}
   \State Set $Q:=I_m$, $R:=I_m$. Call \Call{von Neumann}{$A,Q,\varepsilon$} to obtain $(x,y)$.
 \While{$A^\T Qy\not>0$}
    \State {\bf rescale}
   \[R:=\frac{1}{1+\varepsilon}\left(R+\sum_{i=1}^n \frac{x_i}{\|a_i\|_Q^2} a_ia_i^\top \right);\quad  Q:=R^{-1}.\]
    \State Call \Call{von Neumann}{$A,Q,\varepsilon$} to obtain $(x,y)$.
 \EndWhile
\Return{The feasible solution $\bar y:=Qy$ to \eqref{eq:main dual}.}
 \end{algorithmic}
\caption{Full Support Image Algorithm}\label{alg:dual-alg-full}
\end{algorithm}

\end{minipage}
\end{center}
\end{figure}

\begin{theorem}\label{thm:main-matrix-dual} For any input matrix $A\in\R^{m\times
    n}$ such that $\rk(A)=m$ and \eqref{eq:main dual} is feasible,
Algorithm~\ref{alg:dual-alg-full} finds a
feasible solution to \eqref{eq:main dual} by performing
$O\left(m^3\log{\rho_A^{-1}}\right)$ von Neumann iterations. The
total number of arithmetic operations is
$O\left(m^2n^2\log{\rho_A^{-1}}\right)$.
\end{theorem}
This will be proved in Section~\ref{sec:dual-full-analysis}.
Using Lemma~\ref{lem:numerical-shit}, we obtain the running time in terms of the
encoding length $L$.

\begin{corollary}\label{cor:bit complexity dual} Let $A\in\Z^{m\times n}$ be an
  integer matrix of encoding size $L$. If
  $\rk(A)=m$ and \eqref{eq:main dual} is feasible,
then Algorithm~\ref{alg:dual-alg-full} finds a feasible
  solution of~\eqref{eq:main dual} in $O\left(m^2n^2L\right)$ arithmetic operations.
\end{corollary}

In the above framework,  the only important property of
von Neumann's algorithm is that it delivers a vector
$y\in\conv(a_1/\|a_1\|_Q,\ldots,a_n/\|a_n\|_Q)$ with $\|y\|\leq
O(1/m)$ in time polynomial in $m$ and $n$.
This can be also achieved using other first order
methods, such as Perceptron, the DV-updates, or Wolfe's nearest-point
algorithm~\cite{Wolfe}. The best running times can be obtained using the Smoothed Perceptron algorithm of Pe\~na and Soheili \cite{Pena-Soheili-smooth} or the  Mirror Prox for Feasibility Problems (MPFP) by Yu et
 al. \cite{Yu-Karzan-Carbonell}.

\begin{theorem}\label{thm:main-matrix-smoothed} For any input matrix
$A\in\R^{m\times n}$ such that $\rk(A)=m$ and \eqref{eq:main dual} is feasible,  Algorithm~\ref{alg:dual-alg-full} with
the Smoothed Perceptron of \cite{Pena-Soheili-smooth} or the  Mirror Prox method of \cite{Yu-Karzan-Carbonell} finds a
feasible solution to \eqref{eq:main dual} by performing
$O\left(m^2\sqrt{\log n}\cdot\log \rho_A^{-1}\right)$ iterations. The
number of arithmetic operations is $O\left(m^3 n \sqrt{\log n}\cdot
  \log \rho_A^{-1}\right)$.
If $A\in\Z^{m\times n}$ is integer with encoding length $L$, then the running time is $O\left(m^3 n \sqrt{\log n}\cdot L\right)$.
\end{theorem}

\paragraph{Comparison to previous work}
The main difference between Algorithm~\ref{alg:dual-alg-full} and the
algorithms by Betke's \cite{Betke} or by Pe\~na and
Soheili's \cite{Pena-Soheili} is the use of a multi-rank rescaling, as opposed to rank-1 updates.
The multi-rank rescaling allows for a factor $n$ improvement in the overall number of iterations. While we use a
similar volumetric potential, the multi-rank update guarantees a constant factor decrease in potential (Lemma~\ref{lem:dual-vol-dec})
whenever  in the algorithm $\|y\|_Q\in O(1/m)$, whereas the rank-1 update
provides the same guarantee only when $\|y\|_Q\in O(1/(m\sqrt{n}))$.

\subsection{Analysis}\label{sec:dual-full-analysis}
It is easy to see that the matrix $R$ remains positive semidefinite throughout the algorithm,
and admits the following decomposition.

\begin{lemma}\label{lem:gamma-form}
At any stage of the algorithm, we can write the matrix $R$ in the form
\[
R=\alpha I_m+\sum_{i=1}^n \gamma_i \hat{a}_i\hat{a}_i^\T
\]
where $\alpha=1/(1+\varepsilon)^t$ for the total number of rescalings
$t$ performed thus far, and $\gamma_i\ge 0$. The
trace is $\tr(R)=\alpha m+\sum_{i=1}^n\gamma_i$.
\end{lemma}

Recall that we denote by $\Sigma_A=\{y\in \R^m: A^\T y\ge 0\}$ the image cone. Let us define the set
\begin{equation}
\label{def:FA}
F_A=\Sigma_A\cap \B^m.
\end{equation}

The ellipsoid
$E(R)=\{z\in \R^m: \|z\|^2_{R}\le 1\}$
plays a key role in the analysis, due to the following properties.

\begin{lemma}\label{lem:F-in}
Throughout Algorithm~\ref{alg:dual-alg-full}, $F_A\subseteq E(R)$ holds.
\end{lemma}

\begin{proof}
The proof is by induction on the number of rescalings. At
initialization, $F_A\subseteq E(I_m)=\B^m$ is trivial. Assume
$F_A\subseteq E(R)$, and we rescale $R$ to $R'$. We show
$F_A\subseteq E(R')$. Consider an arbitrary point $z\in F_A$;
then $ a_i^\T z\ge 0$ for all $i\in[n]$ and, by the induction hypothesis, $\|z\|^2_{R}\le 1$ because $z\in E(R)$.

For the vector $x$ returned by the von Neumann algorithm, the
algorithm sets
\[
R'=\frac{1}{1+\varepsilon}\left(R+\sum_{i=1}^n \frac{x_i }{\|a_i\|_Q^2}a_i a_i^\T\right).
\]
Recall that, in the algorithm, the vector $y=\sum_{i=1}^n x_i \frac{a_i}{\|a_i\|_Q}$ satisfies $\|y\|_Q\leq\varepsilon$.
By the Cauchy-Schwartz inequality, we have $y^\T z=y^\T Q^{1/2}{Q^{-1/2}}z\leq \|y\|_Q\|z\|_R\leq \varepsilon$, and similarly $ a_i^\T z\le \|a_i\|_Q\|z\|_R\leq \|a_i\|_Q$ for every $i\in[m]$. We then have
\begin{eqnarray*}
\|z\|^2_{R'}&=&\frac{1}{1+\varepsilon} z^\T \left(R+\sum_{i=1}^n\frac{x_i}{\|a_i\|_Q^2} a_ia_i^\T\right) z
=\frac{1}{1+\varepsilon}\left(\|z\|^2_{R}+\sum_{i=1}^n x_i \left(\frac{a_i^\T z}{\|a_i\|_Q} \right)^2\right)\\
&\le&\frac{1}{1+\varepsilon}\left(1+\sum_{i=1}^n x_i \frac{a_i^\T z}{\|a_i\|_Q} \right) = \frac{1+ y^\T z}{1+\varepsilon}\leq 1,
\end{eqnarray*}
where the first inequality follows from the facts that $\|z\|_R\leq 1$, $x\geq 0$, and $0\le a_i^\T z\le \|a_i\|_Q$ for all $i\in[n]$, while the second follows from  $y^\T z\leq \varepsilon$. Consequently, $z\in E(R')$, completing the proof.
\qued \end{proof}

\begin{lemma}\label{lem:dual-vol-dec}
$\det(R)$ increases by a factor at least $16/9$ at
every rescaling.
\end{lemma}

\begin{proof}
Let $R$ and $R'$ denote the matrix before and after the rescaling. Let $X=\sum_{i=1}^n
  x_ia_ia_i^\T/\|a_i\|_Q^2$; hence $R'={(R+X)}/(1+\varepsilon)$. The ratio of the
  two determinants is
\[
\frac{\det(R')}{\det(R)}=\frac{\det(R+X)}{(1+\varepsilon)^m\det(R)}=\frac{\det\left(I_m+R^{-1/2}XR^{-1/2}\right)}{(1+\varepsilon)^m}
\]
Now $R^{-1/2}=Q^{1/2}$, and $Q^{1/2}XQ^{1/2}$ is a positive
semidefinite matrix. The determinant can be lower bounded using
using Lemma~\ref{lem:linalg}(\ref{det-trace-1}) and the linearity of
the trace.
\[
\frac{\det(R')}{\det(R)}\ge \frac{1+\tr(Q^{1/2}XQ^{1/2})}{(1+\varepsilon)^m}=
\left(1+\sum_{i=1}^n \frac{x_i}{\|a_i\|_Q^2} \tr(Q^{1/2}a_ia_i^\T Q^{1/2})\right)/(1+\varepsilon)^m.
\]
Finally,
$\tr(Q^{1/2}{a_ia_i^\T }Q^{1/2})=\tr(a_i^\T Qa_i)=\|a_i\|_Q^2$. Therefore
we conclude
\[
\frac{\det(R')}{\det(R)}\ge \frac{1+\sum_{i=1}^n x_i}{(1+\varepsilon)^m}=\frac{2}{(1+\varepsilon)^m}.
\]
The claims follows using that $\varepsilon=\const$.
\qued \end{proof}

 We now present the proofs of
Theorems~\ref{thm:main-matrix-dual} and \ref{thm:main-matrix-smoothed} based on these lemmas.

\begin{proof}[Proof of Theorem~\ref{thm:main-matrix-dual}]

By Lemma~\ref{lem:rho}, $\Sigma_A$ contains a ball $B$ of radius $\rho_A$ centered on the surface of the unit sphere.
Consequently, $B\subseteq \Sigma_A\cap (1+\rho_A) \B^m=(1+\rho_A)F_A$.
In particular, $F_A$ contains a ball of radius $\rho_A/(1+\rho_A)$, therefore
$\vol(F_A)\geq (\rho_A/(1+\rho_A))^m \vol(\B^m)\geq (\rho_A/2)^m \vol(\B^m)$
On the other hand, since $\vol(E(R))=\det(R)^{-1/2}\vol(\B^m)$,
Lemma~\ref{lem:dual-vol-dec} implies that $\vol(E(R))$ decreases at
least  by a factor  $2/3$ at every rescaling.
Lemma~\ref{lem:F-in} ensures that $\vol(E(R))\geq \vol(F_A)$.
Consequently, the total number of rescalings during the entire course of the algorithm  provides the bound $O(m\log\rho_A^{-1})$.

By Lemma~\ref{lemma:VN running time}, the von Neumann algorithm performs
$O(m^2)$ iterations between two consecutive rescalings, where each von Neumann
iteration can be implemented in time $O(n)$ assuming that we compute
the matrix $A^\T Q A$ at the beginning and after every
rescaling. Thus the total number of arithmetic operations required by the von Neumann iterations between two rescalings is $O(m^2n)$. To compute $A^\T Q A$, provided  we have computed $Q$, requires time $O(n^2m)$. Updating the matrix $R$ requires time $O(m^2n)$, since we need time $O(m^2)$ to compute each of the $n$ terms $x_i a_i a_i^\T/\|a_i\|_Q^2$, $i\in[n]$. The inverse $Q$ of $R$ can be computed in time $O(m^3)$. Hence the overall number of arithmetic operations needed between two rescalings is $O(n^2m)$. This gives an overall complexity of $O(n^2m^2\log \rho_A^{-1})$ arithmetic operations.
\qued \end{proof}
We note that the higher running time compared to Algorithm~\ref{alg:primal-algorithm-matrix} is
due to the time required to update $A^\top QA$. This has to be
recomputed from scratch; whereas the corresponding update to $A^\top
A$ in the kernel case was done in $O(n^2)$, since a rank-1 rescaling
was used.

\begin{proof}[Proof of Theorem~\ref{thm:main-matrix-smoothed}]
Both the Smoothed Perceptron of \cite{Pena-Soheili-smooth} or the  Mirror Prox method of \cite{Yu-Karzan-Carbonell}
terminate in $O(\sqrt{\log n}/\varepsilon)$ iterations with output $x\in\R^n_+$ and $y\in \R^m$ such that $\|x\|_1=1$, $y=\sum_{i=1}^n x_i a_i/\|a_i\|_Q$, and either
$A^\T Q y>0$, or  $\|y\|_Q\le\varepsilon$. For both methods, each iteration requires $O(mn)$ arithmetic operations.
As before, $R$ and $Q$ can be recomputed in time $O(m^2 n)$. Thus the overall number of operations required between rescalings is $O(m^2n\sqrt{\log n})$.
As before the total number of rescalings is $O(m\log\rho_A^{-1})$. These together give
the claimed bound.
\qued \end{proof}

\subsection{Oracle model for strict conic feasibility}

Observe that Algorithm~\ref{alg:dual-alg-full} does not require explicit knowledge of the matrix $A$.
In particular, Algorithm~\ref{alg:dual-alg-full} can be easily adapted to an oracle
model. Here the purpose is to find a point in the interior of a full dimensional cone defined as $\Sigma=\{y\in\R^m\st a_i^\T y\geq 0 \; \forall a_i\in I\}$,
where $I$ is a set (possibly infinite) indexing vectors $a_i\in\R^m$, $i\in I$.
We assume that we have access to  a {\em strict separation oracle} SO,
where for each $v\in \R^m$ the call SO$(v)$ returns `YES' if $v\in \mathrm{int}(\Sigma)$ (that is, if $a_i^\T v> 0$ for all $i\in I$),
or it returns $a_k$ for some $k\in I$ such that $a_k^\T v\leq 0$.

Below we present Algorithm~\ref{alg:conic} to determine a point in the interior of $\Sigma$.
The algorithm is nearly identical to Algorithm~\ref{alg:dual-alg-full}, and it uses an oracle version of von Neumann (Algorithm~\ref{alg:orac-Neumann}).
The running time is expressed in terms of the Goffin measure $\rho_\Sigma$ of a full-dimensional cone $\Sigma$,
which is the radius of the largest ball contained in $\Sigma$ centered
on the surface of the unit sphere,
\[\rho_\Sigma\eqdef\sup\{r\st \B^m(p,r)\subseteq \Sigma\;\exists\, p\in\R^m\mbox{ s.t. } \|p\|=1\}.\]

\begin{figure}[htb!]
\begin{center}
\begin{minipage}{0.85\textwidth}
\begin{algorithm}[H]
\raggedright
  \begin{algorithmic}[1]
    \Require{A positive definite
      matrix $Q\in \R^{m\times m}$ and an $\varepsilon>0$.}
    \Ensure{Vectors $\{a_i: i\in N\}$ for $N\subseteq I$, $x\in \R_+^N$, $y$, such that
      $y=\sum_{i\in N}{x_i a_i/\|a_i\|_Q}$, $\sum_{i\in N} x_i=1$, and
      either  $Qy\in \mathrm{int}(\Sigma)$ or  $\|y\|_Q\le \varepsilon$.}
    \State Call SO$(0)$ to obtain  $a_k$. Set $N:=\{k\}$, $x_k:=1$, $y:=a_k/\|a_k\|_Q$.
    \While{$\|y\|_Q>\varepsilon$}
        \If{SO$(Qy)$ returns `YES'} \Return{$\{a_i: i\in N\}$, $x$, $y$  }
        \State Terminate.
    \Else{ let $a_k$, $k\in I$, be the output of SO$(Qy)$;}
      \State Let $\lambda:=\displaystyle\frac{\scalar{y-a_k/\|a_k\|_Q,y}_Q}{\|y-a_k/\|a_k\|_Q\|^2_Q}$;
       \State {\bf update} $x_i:=(1-\lambda)x_i$ for all $i\in N\sm\{k\}$, $x_k:=(1-\lambda)x_k+\lambda$;
       \footnote{For notational convenience, we consider $x_k$ to be $0$ if $k\notin N$.}
       \State \quad $\displaystyle y:=(1-\lambda)y+\lambda\frac{a_k}{\|a_k\|_Q}$, $N:=N\cup\{k\}$;
   \EndIf
	\EndWhile
\State{\Return{$\{a_i: i\in N\}$, $x$, $y$.  }}
 \end{algorithmic}
\caption{Oracle von Neumann algorithm}\label{alg:orac-Neumann}
\end{algorithm}
\end{minipage}
\end{center}
\end{figure}

\renewcommand{\algorithmicrequire}{\textbf{Input:}}
\renewcommand{\algorithmicensure}{\textbf{Output:}}
\begin{figure}[h!]
\begin{center}
\begin{minipage}{0.85\textwidth}
\begin{algorithm}[H]
\raggedright
  \begin{algorithmic}[1]
    \Ensure{A point in the interior of a full dimensional cone
      $\Sigma$ given via a separation oracle SO.}
   \State Set $Q:=I_m$, $R:=I_m$. Call \Call{Oracle von
     Neumann}{$Q,\varepsilon$} to obtain $\{a_i: i\in N\}$, $x\in\R^N$, $y\in\R^m$.
 \While{$Qy\not\in\mathrm{int}(\Sigma)$}
    \State {\bf rescale}
   \[R:=\frac{1}{1+\varepsilon}\left(R+\sum_{i\in N} \frac{x_i}{\|a_i\|_Q^2} a_ia_i^\top \right);\quad  Q:=R^{-1}.\]
    \State Call \Call{Oracle von Neumann}{$Q,\varepsilon$} to obtain
    $\{a_i: i\in N\}$, $x$, $y$.
 \EndWhile
\State \Return{$\bar y:=Qy$.}
 \end{algorithmic}
\caption{Strict Conic Feasibility Algorithm}\label{alg:conic}
\end{algorithm}

\end{minipage}
\end{center}
\end{figure}

\begin{theorem}\label{thm:conic}
For any full dimensional cone $\Sigma$ expressed by a separation oracle,
Algorithm~\ref{alg:conic} finds a point in $\mathrm{int}{(\Sigma)}$ by performing
$O\left(m^3\log{\rho_\Sigma^{-1}}\right)$ von Neumann iterations.
The total number of oracle calls is $O\left(m^3\log{\rho_\Sigma^{-1}}\right)$,
while the total number of arithmetic operations is $O\left(m^5\log{\rho_\Sigma^{-1}}\right)$.
\end{theorem}
\begin{proof}
The analysis is nearly identical to that of Theorem~\ref{thm:main-matrix-dual}.
As before, Algorithm~\ref{alg:orac-Neumann} terminates in at most $\ceil{1/\varepsilon}\in O(m^2)$ iterations and the number of rescalings
in Algorithm~\ref{alg:conic} is $O\left(m\log{\rho_\Sigma^{-1}}\right)$. Thus we perform $O\left(m^3\log{\rho_\Sigma^{-1}}\right)$ von Neumann iterations,
each requiring one oracle call.
For the number of arithmetic operations,
we observe first that the set $N$ computed by Algorithm~\ref{alg:orac-Neumann} has size at most $\ceil{1/\varepsilon}\in O(m^2)$,
since $|N|$ increases by at most one at every iteration.
In every von Neumann iteration, recomputing $x$ requires $O(|N|)=O(m^2)$ arithmetic operations,  recomputing $y$ requires $O(m)$ operations, while recomputing $Qy$ requires $O(m^2)$ operations, thus overall these computations require $O\left(m^5\log{\rho_\Sigma^{-1}}\right)$ arithmetic operations.
Recomputing $R$ during rescalings requires $O(m^2|N|)=O(m^4)$ operations, while recomputing $Q=R^{-1}$ requires $O(m^3)$ operations,
for a total of $O\left(m^5\log{\rho_\Sigma^{-1}}\right)$ arithmetic operations over all rescalings.
\end{proof}

\section{Maximum support algorithms}\label{sec:max support}

In the case $\rho_A=0$, the volumetric arguments of the previous sections fail: both sets $P_A$ and $F_A$ are lower dimensional and therefore have volume 0. In what follows, we show how both algorithms naturally extend to this scenario.

Given any linear subspace $H$, we denote by $\supp(H_+)$ the {\em maximum support} of $H_+$, that is, the unique inclusion-wise maximal element of the family $\{\supp(x)\st x\in H_+\}$. Note that, since $H_+$ is closed under summation, it follows that $\supp(H_+)=\{i\in[n]\st\exists x\in H_+\; x_i>0\}$.
We denote
\begin{equation}\label{eq:support} S^*_A\eqdef \supp(\ker(A)_+),\quad
  T^* _A\eqdef \supp(\im(A^\T)_{+}).\end{equation}
When clear from the context, we will use the simpler notation $S^*$
and $T^*$.
Since $\ker(A)$ and $\im(A^\T)$ are orthogonal to each other, it is immediate that $S^*\cap T^*=\emptyset$. Furthermore, the  strong
duality theorem implies that $S^*\cup T^*=[n]$.

The Maximum Support Kernel Algorithm (Section~\ref{sec:kernel-max})  finds a
solution $x$ to \eqref{eq:primal} with $\supp(x)=S^*$, and the Maximum Support
Image Algorithm (Section~\ref{sec:image-max}) finds a solution $y$ to
\eqref{eq:dual} with $\supp(A^\top y)=T^*$.
In this section, we show that these algorithms can be directly obtained using
the full support algorithms, by repeatedly removing vectors from the support based
on their lengths after a sequence of rescalings. With this direct implementation
however, the maximum support algorithm runs the corresponding full support
algorithm $n$ times in the kernel case and $m$ times in the image case, leading
to an increase in running time. 

With small modifications, both maximum support
algorithms can be implemented in essentially the same asymptotic running time as
their full support counterparts. We defer these variants to
Appendix~\ref{sec:app-max}, since the amortized
analyses are somewhat technical. Still, they  offer some  interesting insights for
degenerate LPs. In particular, they show how to bound the degradation of an
ellipsoidal outer approximation of the feasible region when moving to a lower
dimensional space, which we believe to be of independent interest.

We will need to argue about lower dimensional ellipsoids and their volumes. Let $Q\in\bb{S}_{++}^d$.
For a linear subspace $H\subseteq \R^d$, we let $E_H(Q)\eqdef E(Q)\cap
H$. Further, we define the projected determinant of $Q$ on $H$ as
\[
\det_H(Q)\eqdef \det(W^\top Q W),
\]
where $W$ is any matrix whose columns form an orthonormal basis of
$H$. Note that the definition is independent of the choice of the
basis $W$. Indeed, if $H$ has dimension $r$, then
\begin{equation}\label{formula:ellipsoid-volume}
\vol_r(E_H(Q))=\frac{\nu_r}{\sqrt{\det_H(Q)}}.
\end{equation}

\subsection{The Maximum Support Kernel Algorithm}\label{sec:kernel-max}
We start with an easy observation; the proof is deferred to  Appendix~\ref{app:proofs}.
\begin{restatable}{lemma}{symmetrized}\label{lem:X-S} Let
  $A\in\R^{m\times n}$ and $S^*=S^*_A$. Then
$\spn(P_A)=\im(A_{S^*})$, and $P_A=P_{A_{S^*}}$.
\end{restatable}
In this section we observe that, if Algorithm~\ref{alg:primal-algorithm-matrix} is applied to a matrix $A$ with $\rho_A\geq 0$ (that is, \eqref{eq:main problem} is infeasible), then after a certain number of iterations,
based on the encoding size of $A$, we can establish  a column of $A$ that cannot be contained in $S^*_A$.
This is based on the observation contained in the next lemma that columns in $S^*_A$ need to
 remain ``short'' throughout the execution of Algorithm~\ref{alg:primal-algorithm-matrix}.

\begin{lemma}\label{lem:contained-in-ellipsoid}
Let $A\in\R^{m\times n}$ such that $\|a_i\|= 1$ for all $i\in
[n]$. Let $H=\im(A)$. After $t$ rescalings in Algorithm~\ref{alg:primal-algorithm-matrix} with $A$ as input, let $A'$ be the current matrix. Let $M\in\R^{m\times m}$ be the matrix obtained by combining all $t$ rescalings, so that $A'=MA$, and define $Q=(1+3\varepsilon)^{-2t}M^\T M$. The following hold.
\begin{enumerate}[(i)]
\item $P_A\subseteq E(Q)$,
\item $\|a_i\|_Q\le |\rho_{A_{S^*}}|^{-1}$ for every $i\in S^*$,
\item If $\mu:=\max_{i\in[n]}{\|a_i\|_Q}$, then $(\mu^{-1}\cdot|\rho_{(A,-A)}|)\bb{B}^m\cap H\subseteq E_H(Q)$,
\item At every rescaling, the relative volume of $E_H(Q)$ decreases by a factor at least $2/3$.
\end{enumerate}
\end{lemma}
\begin{proof} {\bf (i)} At intialization $Q=I_m$, so $P_A\subseteq E(Q)=\bb{B}^m$. After $t$ rescalings, by Lemma~\ref{lem:volume-increase} applied $t$ times, $M P_A\subseteq (1+3\varepsilon)^tP_{A'}$. Recall that by definition $P_{A'}\subseteq \bb{B}^m$, thus  $P_A\subseteq (1+3\varepsilon)^t M^{-1}\bb{B}^m=E(Q)$.

\noindent {\bf (ii)}  Lemma~\ref{lem:X-S} shows that $P_A=P_{{A_{S^*}}}$. Hence, by  Lemma~\ref{lem:rho} and the fact that $\|a_i\|_2=1$ for all $i\in [n]$, it follows that $|\rho_{A_{S^*}}|a_i\in P_{A}$ for all $i\in S^*$. From (i), $|\rho_{A_{S^*}}|a_i\in E(Q)$, which implies $|\rho_{A_{S^*}}|\|a_i\|_Q\leq 1$.

\noindent{\bf (iii)}  By definition, $\mu^{-1}  a_j\in E(Q)$, therefore $E_H(Q)$ contains $\mu^{-1}\cdot\conv(A,- A)$. Note that $\rho_{(A,-A)}<0$, hence  Lemma~\ref{lem:rho} implies that  $|\rho_{(A,-A)}|\bb{B}^m\cap H\subseteq \conv(A,-A)$. This implies $(\mu^{-1}\cdot\rho_{(A,-A)})\bb{B}^m\cap H\subseteq E_H(Q)$ as needed.

\noindent{\bf (iv)} Let $r=\rk(A)$. Rescaling corresponds to replacing  $M$ by $M'=(I_m+\hat y\hat y^\top)M$, where $y=A'x$ is the current point computed by the algorithm. Accordingly, matrix $Q$ is replaced by $Q'=(1+3\varepsilon)^{-2}M^\T M$. Since $y\in H$, the function $z\to (I_m+\hat y\hat y^\top)z$ is the identity over $H^\bot$ and it is an automorphism over $H$. This implies that  $E_H(Q') =(1+3\varepsilon) (I_m+\hat y\hat y^\top)^{-1} E_H(Q)$ and that $\vol_r(E_H(Q'))=(1+3\varepsilon)^r\det(I_m+\hat y\hat y^\top)^{-1} \vol_r(E_H(Q))$. The statement follows from the fact that $(1+3\varepsilon)^{r}\leq 4/3$ and $\det(I_m+\hat y\hat y^\top)=2$.
\qued \end{proof}

\paragraph{Basic Maximum Support Kernel Algorithm.}
Based on the previous lemma, we can immediately describe an algorithm for the
maximum support kernel problem for a matrix $A\in \Z^{m\times n}$ of encoding
size $L$. Let $\bart:=\bart_A$ defined in \eqref{def:delta}, recalling that $\bart \ge 2^{-4L}$.
Let $S^* := S^*_A$.  Observe that by
Lemma~\ref{lem:numerical-shit} we have $|\rho_{A_{S^*}}| \geq \bart$ if $S^*\neq
\emptyset$ and $|\rho_{(A_S,-A_S)}|\geq \bart$ for any $\emptyset \neq S
\subseteq [n]$ (indeed, note that $\Delta_A\geq \Delta_{A_S}=\Delta_{(A_S,-A_S)}$ by definition).

We start with initial guess $S = [n]$ for the support $S^*$. To get a max support
solution we will iteratively run a slightly modified version of
Algorithm~\ref{alg:primal-algorithm-matrix} on the matrix $\hat A_S$, which will
either return a full support solution $\hat{A}_S x = 0$, $x > 0$, or an index $k
\in S\sm S^*$. In the former case, we return $x$ with zeros on the
components in $[n] \setminus S$ as a max support solution. In the latter case,
we replace $S := S \setminus \set{k}$ and rerun the modified
Algorithm~\ref{alg:primal-algorithm-matrix} on $\hat{A}_S$. We continue this
process until either $S = \emptyset$ or a solution is found.

For the modified Algorithm~\ref{alg:primal-algorithm-matrix} on $\hat{A}_S$, the
only change is a step which recognizes when an index in $S$ is not in the
support $S^*$. For this purpose, we maintain the vector lengths
$\|\hat a_i\|_Q$ as in
in Lemma~\ref{lem:contained-in-ellipsoid} applied to $\hat A_S$. After each
rescaling, which updates $Q$, we simply check if there is an index $k \in S$
such that $\|\hat{a}_k\|_Q > \bart^{-1}$ (here $\hat{a}_k$ refers the normalized
column of the original matrix $A$), and if so, we return it as an index not in $S^*$. Note that this
assertion is justified by
Lemma~\ref{lem:contained-in-ellipsoid}(ii). Let us note that if $A'$
is the current matrix in Algorithm~\ref{alg:primal-algorithm-matrix}
on $\hat{A}_S$
after $t$ rescalings, then $\|\hat
a_i\|_Q=\|a'_i\|/(1+3\varepsilon)^t$. Therefore, we do not need to
maintain the matrix $Q$ explicitly. 

\medskip

We now bound the running time of the modified
Algorithm~\ref{alg:primal-algorithm-matrix} at every call. Let $S \supseteq S^*$
be the current support, $H = \im(\hat{A}_S)$ and $r := \rk(A_S) \leq m$. Note
that as long as we have not identified a column to remove, which would end the
current call on $\hat{A}_S$, by part (iii) of the above lemma we have that
$\bart^{2}\bb{B}^m \cap H \subseteq E_H(Q)$, hence $\vol_r(E_H(Q)) \geq
\bart^{2r} \nu_r$. Since initially $\vol(E_H(Q))_r= \nu_r$, by part (iv) we
conclude that the number $K$ of rescalings is bounded by $O(\log(\bart^{2r}))$,
hence $K\in O(mL)$ (since $r \leq m$). By
Lemma~\ref{lem:delta-rounding} the call to Algorithm~\ref{alg:primal-algorithm-matrix} terminates as soon as the current
vector $y$ has norm less than $\bart \leq |\rho_{(A_S,-A_S)}|$, hence the number of
DV iterations can be bounded exactly as in the proof of
Theorem~\ref{thm:main-matrix} by $O(m^2(n+K+\log(\bart^{-1})))=O(m^3L)$. The
proof of Theorem~\ref{thm:main-matrix} also shows that each DV update can be
performed in time $O(n)$ and each rescaling can be computed in time $O(n^2)$.
Hence each call to Algorithm~\ref{alg:primal-algorithm-matrix} requires
$O((m^3n+mn^2)L)$ arithmetic operations, so that the overall number of
operations to compute a maximum support solution is
$O((m^3n^2+mn^3)L)$.


\subsection{The Maximum Support Image Algorithm}\label{sec:image-max}

In this section we show that if Algorithm~\ref{alg:dual-alg-full} is applied to a
matrix $A$ with $\rho_A \leq 0$ (that is, \eqref{eq:main dual} is infeasible), then after a certain number of iterations,
based on the encoding size of $A$, we can pinpoint an index $k \in
[n]\sm T^*_A$. 

The following will be a key concept in the analysis. Given a convex set $X\subset \R^d$ and a vector $a\in \R^d$, we define the {\em width of $X$ along $a$} as
\begin{equation}\label{eq:set width}\width_X(a)\eqdef \max  \{a^\T z\st z\in X\}.\end{equation}

\begin{lemma}\label{lemma:ellipsoid width}
Given $R\in\bb{S}^d_{++}$, let $E:=E(R)$. For any  $a\in\R^d$,  $\width_E(a)=\|a\|_{R^{-1}}$.
\end{lemma}
\begin{proof}
For every $z\in E$, $a^\T z=a^\T R^{-1/2}R^{1/2}z\leq \|a\|_{R^{-1}}\|z\|_{R}\leq \|a\|_{R^{-1}}$ where the first inequality follows from the Cauchy-Schwarz inequality and the second from $z\in E$. On the other hand, if we define $z=R^{-1}a/\|a\|_{R^{-1}}$, it follows that $z\in E$ and $a^\T z=\|a\|_{R^{-1}}$.
\qued \end{proof}

Let us introduce
\begin{equation}
{\omega}_A\eqdef\min_{i\in T_A^*}
\width_{F_A}(\hat a_i), \label{def:omega}
\end{equation}
The
quantity $\omega_A$ is related to $\rho_{A}$, as illustrated by
the next claim,  whose straightforward proof can be found in Appendix~\ref{app:proofs}.
We recall that $\bart_A$ was defined in \eqref{def:delta}.

\begin{restatable}{Claim}{omegab}\label{cl:omega-bound}
If $T_A^*=[n]$, then $\omega_A\ge \rho_{A}$. If $T_A^*\neq \emptyset$ and $A$ has integer entries, then $\omega_A\ge \bart_A$.
\end{restatable}

The next lemma provides the main tools to detect columns $k\notin T_A^*$ in the
Full Support Image Algorithm. We recall that $F_A$ is as defined in~\eqref{def:FA}.

\begin{lemma}\label{lem:lower-width-2}
Let $R\in\bb{S}_{++}^m$ such that $F_A\subseteq E(R)$, and let $Q=R^{-1}$. For $k\in[n]$, if $\|\hat a_k\|_Q< \omega_A$ then $k\notin
T_A^*$.
\end{lemma}
\begin{proof}
From Lemma~\ref{lemma:ellipsoid
  width} and $F_A\subseteq E(R)$, we have
$
\|\hat a_k\|_Q={\width_{E(R)}(\hat a_k)}\ge
{\width_{F_A}(\hat a_k)}\ge \omega_A
$.
\qued
\end{proof}

Lemma~\ref{lem:dual-vol-dec} shows that in Algorithm~\ref{alg:dual-alg-full}, $\det(R)$ increases at least by a factor
$16/9$ in every rescaling.
The following lemma bounds $\min_{k\in [n]}\|\hat
a_k\|_Q$ in terms of $\det(R)$.
\begin{lemma}\label{lem:a-k-Q-bound}
At any stage of Algorithm~\ref{alg:dual-alg-full} applied to $A \in \R^{m \times
n}$, if $\det(R)>1$, then there exists $k\in [n]$ such that
$\|\hat a_k\|_Q\le {(\det(R)^{1/m}-1)}^{-1/2}$.
\end{lemma}
\begin{proof}
Let $k=\arg\min_{i\in T}\|\hat a_i\|_Q$. Let us use the decomposition of $R$ as in
Lemma~\ref{lem:gamma-form}. Then
\begin{align}\label{eq:a_k bound}
\|\hat a_k\|^2_Q\sum_{i=1}^n\gamma_i &\leq \sum_{i=1}^n\gamma_i \|\hat a_i\|^2_Q=\sum_{i=1}^n\gamma_i
\left(\hat a_i^\T Q \hat a_i\right)
= \tr\left(Q \sum_{i=1}^n\gamma_i
  \hat a_i\hat a_i^\T\right)\\
&= \tr(Q(R-\alpha I_m))=\tr(I_m-\alpha Q)=m-\alpha \tr(Q)<m~. \nonumber
\end{align}
The third equality used the decomposition of $R$, the fourth used $QR=I_m$, and the final inequality holds since $Q$
is positive definite.

The fact that $\tr(R)=\alpha m+\sum_{i=1}^n\gamma_i\le m+\sum_{i=1}^n\gamma_i$ and Lemma~\ref{lem:linalg}(\ref{det-trace-2}) imply that $\sum_{i=1}^n\gamma_i\geq \tr(R)-m\ge m(\det(R)^{1/m}-1)$. Note that the latter term is positive because $\det(R)>1$, therefore the statement follows from \eqref{eq:a_k bound}.
\qued \end{proof}

\paragraph{Basic Maximum Support Image Algorithm} In light of the
above lemmas, we can extend the Full
Support Image Algorithm (Algorithm~\ref{alg:dual-alg-full}) to the maximum
support case for a matrix $A\in \Z^{m\times n}$ of encoding size
$L$ as follows. Let us assume that $rk(A)=m$.
We again use $\bart_{ A}$ as in
\eqref{def:delta} and observe that by Claim~\ref{cl:omega-bound}, $\omega_A\ge\bart_A$ whenever $T^*\neq\emptyset$.
We run Algorithm~\ref{alg:dual-alg-full} until either we can find
$y$ such that $A^\top Qy>0$, or we find an index $k$ such that  $\|\hat
a_k\|_Q< \bart_A$.

Lemmas~\ref{lem:dual-vol-dec} and \ref{lem:a-k-Q-bound} guarantee that
either outcome is reached within $O(mL)$
rescalings. In the first case,  we terminate with the maximum
support solution $Qy$.
 Lemma~\ref{lem:lower-width-2} guarantees that in the
second case, $k\notin T^*$.

Once an index $k\notin T^*$ is identified, we must have $a_k^\top y=0$
for every solution $y$ to \eqref{eq:dual}.
Hence, the necessary update is to project the columns of $A$ onto the
subspace $a_k^\bot$.
 Formally, we compute an orthonormal basis $W\in \R^{m\times (m-1)}$
 of $a_k^\bot$, and replace the matrix $A$ by $A'$ obtained from $W^\top A$ by removing the zero columns. Then, we
 recursively apply the same algorithm to $A'$ instead of $A$. Assume
 we obtain $y'$ as the output from the recursive call, such that
 ${A'}^\top y$ is a maximum support vector in
 $\im({A'}^\top)_{+}$. Then, we output the vector
 $y=Wy'$ for the original matrix $A$.

To verify the correctness of this recursive call, we need to show that
the maximum support solutions to $A$ and $A'$ are in one-to-one
correspondence. Furthermore, we need to provide a lower bound on
$\omega_{A'}$ in terms of $L$. 
 We show that $\omega_{A'}\ge\omega_A$, and therefore
$\bart:=\bart_{A}$ remains a valid lower bound.
These claims are
formally verified in Lemma~\ref{lem:im-dim-red} below.

To estimate the running time of the algorithm, we recall
that a new column outside $T^*$ can be identified within $O(mL)$
rescalings, and there are at most $m$ recursive calls, since every
call decreases the rank of the matrix $A$.
As in the full support algorithm, we can implement
the iterations between two rescalings in $O(n^2m)$ arithmetic
operations. Further, we need to compute orthonormal bases at every
recursive call, which can be done in time $O(r^2)$ for the current
rank $r$. Thus, we obtain a total running time $O(n^2 m^3 L)$.

\begin{lemma}\label{lem:im-dim-red}
Let $A \in \R^{m \times n}$ and  $H \subseteq \R^m$ an $r$-dimensional subspace such that $\Sigma_{A }\subseteq H$. Let $U
\in \R^{m \times r}$ be an orthornormal basis of $H$ and let $A'$ be the
matrix obtained from $U^\top A$ after removing all 0 columns. The following hold:
\begin{enumerate}[(i)]
\item $F_{A} = U F_{A'}$.
\item For $v \in \R^m$, $w = U^\top v$, we have $\width_{F_{A}}(v) =
\width_{F_{A'}}(w)$ and $\width_{F_{A}}(\hat v) \leq \width_{F_{A'}}(\hat w)$.
\item $\omega_{A} \leq \omega_{A'}$.
\end{enumerate}

\end{lemma}
\begin{proof} {\bf(i)}  Take $x\in U F_{A'}$ and $y\in F_{A'}$ such that $x= U y $.
Firstly, $\|x\| = \|Uy\| = \|y\| \leq 1$ since $y \in F_{A'}$. For $i \in [n]$,
note that if $U^\T a_i = 0$ then $a_i^\T x = (U^\T a_i)^\T y = 0$,
and if not, $a_i$ appears as a column of $A'$ and hence $a_i^\T x =
(U^\T a_i) y \geq 0$ since $y \in F_{A'}$. Thus $x \in F_{A}$. For $x \in F_{A}$, since $F_{A}\subseteq H$, we can write $x = U y$ for $y\in \R^r$. Applying
the previous argument in reverse, we conclude that $y \in F_{A'}$ and hence $x \in
U F_{A'}$. Thus $F_{A} = U F_{A'}$ as needed.

\noindent {\bf (ii)} The equality follows directly from part (i) since
\[
\width_{F_{A}}(v) = \width_{UF_{A'}}(v) = \width_{F_{A'}}(U^\top v) = \width_{F_{A'}}(w)\text{.}
\]
We prove the inequality. By positive homogeneity of width, if either $v$
or $w$ equals $0$, we clearly have $0 = \width_{F_{A}}(\hat v) = \width_{F_{A'}}(\hat
w)$ and the statement follows. So we may assume that both $v,w \neq 0$. Since
$U$ is orthonormal, we see that $0 < \|w\| = \|U^\top v\| \leq \|v\|$. Since
$0 \in F_{A}$, by homogeneity we have that
\[
0 \leq \width_{F_{A}}(\hat v) = \frac{1}{\|v\|} \width_{F_{A}}(v)
= \frac{1}{\|v\|} \width_{F_{A'}}(w)
\leq \frac{1}{\|w\|} \width_{F_{A'}}(w) = \width_{F_{A'}}(\hat w)~,
\]
as needed.

\noindent {\bf (iii)}  First, note that the set $T^*_{ A'}$ comprises the indices of
  the columns $U^\top  a_i$ for which $i\in T^*_{A}$, that is,
  $\width_{F_{A'}}(a_i) > 0 $.
The inequality follows from the last statement in part (ii).
\end{proof}


\section{Conclusions}
We have given polynomial time algorithms for the full support and
maximum support versions of the kernel and image problems.
These methods give new insights on how to leverage
the underlying geometry of linear (and more generally conic)
programs. \lnote{removed ellipsoid}

There is an important conceptual difference between the full support and maximum support variants.
The running time of the full support kernel and image algorithms depends on $\log |\rho_A|^{-1}$. However, the algorithms do not require explicit knowledge of $\rho_A$; this parameter only shows up in the running time analysis. These algorithms can be implemented in the real model of computation.

In contrast, the maximum support variants rely on bit complexity
estimations. The algorithms require an integer input matrix, and use $\theta_A$, computed from the Hadamard bound, as a threshold for removing columns from the support.
Given the duality between maximum support versions of~\eqref{eq:main problem}
and~\eqref{eq:main dual}, the most natural goal would be to find a
\emph{complementary pair} of maximum support solutions to~\eqref{eq:primal} and~\eqref{eq:dual}, since such solutions are self-certifying
(i.e.~each would certify that the other is indeed a max support solution).
Developing a rescaling algorithm which solves this problem directly using
natural geometric potentials, as opposed to the bit complexity arguments above,
is an interesting open problem.

We note that the interior point method of
Vavasis and Ye~\cite{vavasis1995,Ye94} provides a complementary pair in the real model of
computation, based on certain condition measures (one of them being
related to our $\omega_A$). However, these condition measures do not improve over the course
of the algorithm. Our goal would be to find an algorithm which
finds a rescaling of the problem that simultaneously approximates both kernel and image
geometries.

\bibliographystyle{abbrv}
\bibliography{rescaled}

\appendix

\section{Faster algorithms for the maximum support problems}\label{sec:app-max}
This Appendix exhibits improved versions of the maximum support kernel and image algorithms described in Section~\ref{sec:max support}. 
The key idea to the amortized analyses is bounding the possible
increase of the volume of the ellipsoidal approximation when moving to
a lower dimensional subspace. The following lemma will be useful in computing the projected
determinant. 
\begin{lemma}\label{lemma:determinant and hyperplane}
  Consider a matrix $R\in \bb{S}^d_{++}$. For a vector $a\in \R^d$, $\|a\|=1$, let $H=\{x\st a^\top x=0\}$. Then
$\displaystyle\det_H(R)=\det(R) \|a\|_{R^{-1}}^2$.
\end{lemma}
\begin{proof}
Let $W\in \R^{d\times(d-1)}$ be a matrix whose columns form an orthonormal basis of $H$.
Since $(W|a)$ is an orthonormal basis of $\R^d$, we have
\begin{eqnarray*}
\det(R)&=&\det\left(\left(\begin{array}{c}W^\top \\a^\top \end{array}\right)R(W|a)\right)=\det\left(\begin{array}{cc}W^\top RW&W^\top Ra\\a^\top RW&\|a\|_R^2\end{array}\right)\\
&=&\det(W^\top RW) (\|a\|_R^2-a^\top RW(W^\top RW)^{-1} W^\top Ra),
\end{eqnarray*}
where the last equality follows from the determinant identity for the Schur's complement. Observe that $\|a\|_R^2-a^\top RW(W^\top RW)^{-1}W^\top Ra=\|q\|^2$, where  $q$ is the orthogonal projection of the vector $v:=R^{\frac{1}{2}} a$ onto the orthogonal complement of the hyperplane $R^{\frac{1}{2}}H$. The orthogonal complement of this hyperplane is the line generated by $p:=R^{-\frac{1}{2}}a$. Thus  $\|q\|=\hat p^\T v=(a^\T R^{-\frac{1}{2}}R^{\frac{1}{2}} a)/\|R^{-\frac{1}{2}}a\|=1/\|a\|_{R^{-1}}$ since $\|a\|=1$. Therefore
$\det_H(R)=\det(W^\top RW)=\det(R)\|a\|_{R^{-1}}^2$, as required.
\qued \end{proof}

\begin{lemma}\label{lemma:ellipsoid and hyperplane}
Let $E\subset \R^d$ be an ellipsoid and $H$ an $r$-dimensional subspace of $\R^d$. Given $a\in H$, $\|a\|=1$, let $H'=\{x\in H\st a^\top x=0\}$.
Then \[\vol_{r-1}(E_{H'})=\frac{\vol_r(E_H)}{\width_E(a)}\cdot
\frac{\nu_{r-1}}{\nu_r}.\]
\end{lemma}
\begin{proof}
We can assume $H=\R^r$, so that $E=E_H$. Let $R\in\bb{S}_{++}^r$ such that $E=E(R)$. The volume of $E$ can be written as
$\vol_{r}(E)=\nu_{r}/\sqrt{\det(R)}$, and using
\eqref{formula:ellipsoid-volume}, we get $\vol_{r-1}(E_{H'})=\nu_{r-1}/\sqrt{\det_{H'}(R)}$. The statement  follows from
Lemmas~\ref{lemma:determinant and hyperplane} and \ref{lemma:ellipsoid width}.
\qued \end{proof}

\subsection{Amortized Maximum Support Kernel Algorithm}\label{sec:app-kernel}

To describe the algorithm, it is more convenient to work with the scalar products defined by $Q$ instead of the rescaled matrix $A'=MA$. Consider a vector $y'=A'x=MAx$ in any iteration of Algorithm~\ref{alg:primal-algorithm-matrix}, and let $y=A x$. Note that $y'=My$, so when computing $\hat a_j'^\T \hat y'$
we have $\hat a_j'^\T \hat y'=\scalar{\frac{a_j}{\|a_j\|_Q},\frac{y}{\|y\|_Q}}_Q$.
Rescaling in Algorithm~\ref{alg:primal-algorithm-matrix} replaces $A'$ by $(I_m+\hat y'\hat y'^\top)A'$,
therefore the corresponding update of the scalar product consists of replacing  $M$ by $(I_m+\hat y'\hat y'^\top)M$ and recomputing $Q$.
Noting that $\|y'\|_2=\|M y\|_2=(1+3\varepsilon)^{t}\| y\|_Q$, the update can be written in terms of $Q$ and $y$, as
\begin{equation}\label{eq:new-Q-kernel}
Q'=\frac{1}{(1+3\varepsilon)^{2(t+1)}}M^\T\Big({I_m+\frac{y' y'^\T}{\|y'\|_2^2}}\Big)\Big({I_m+\frac{y' y'^\T}{\|y'\|_2^2}}\Big)M=\frac{1}{(1+3\varepsilon)^2}\Big(Q+\frac{3Qyy^\T Q}{\| y\|_Q^2}\Big).\end{equation}
We define the procedure \Call{Rescale}{$Q,y$} which, given $Q$ and $y$, replaces $Q$ with the matrix $Q'$ defined in \eqref{eq:new-Q-kernel}.
\bigskip

In this section we show that we can improve the running time estimate of the Basic Maximum Support Kernel Algorithm (Section~\ref{sec:kernel-max}) by a factor $n$ by adopting two ideas.
\begin{itemize}
\item[(a)] Instead of removing a column $a_k$, $k\in T^*$ from $A$ every time we identify one, we maintain as before a set $S\subseteq[n]$ with the property that $S^*\subseteq S$, as well as a set $T\subseteq S$ of indices that we have determined not to belong to $S^*$ (that is, $T\cap S^*=\emptyset$ throughout the algorithm). Whenever we conclude that $i\notin S^*$ for an index $i$ based on Lemma~\ref{lem:contained-in-ellipsoid}(ii), we add $i$ to $T$. Columns indexed by $T$ are removed from $S$ only when doing so decreases the rank of the matrix $A_S$.
\item[(b)] After removing  columns from $A$, instead of restarting from $Q=I_m$ we restart from the same $Q$ we had at the last iteration. If in a given iteration we remove columns from $S$, thus obtaining a set $S'\subsetneq S$, it may happen that the relative volume of $E(Q)\cap\im(A_{S'})$ is larger than the relative volume of $E(Q)\cap\im(A_S)$, but Lemma~\ref{lemma:volume after removal} ensures that the increase in volume is not too large.
\end{itemize}

\renewcommand{\algorithmicrequire}{\textbf{Input:}}
\renewcommand{\algorithmicensure}{\textbf{Output:}}
\begin{figure}[htb]
\begin{center}
\begin{minipage}{0.85\textwidth}
\begin{algorithm}[H]
\raggedright
  \begin{algorithmic}[1]
    \Require{A matrix $A\in\Z^{m\times n}$.}
    \Ensure{A maximum support solution to the system \eqref{eq:primal}.}
\State Compute $\bart:=\bart_A$ as in \eqref{def:delta}.
    \State Compute $\Pi:=\Pi_{\hat A}^K$.
    \State Set $x_j:=1$ for all $j\in [n]$, and $y:=\hat Ax$.
    \State Set $S:=[n]$, $T:=\emptyset$, $Q:=I_m$.
    \While{($S\neq \emptyset$) and ($\Pi x\not> 0$)}
  	\State Let $\displaystyle k:=\arg\min_{i\in S} \scalar{a_i,y}_Q/\|a_i\|_Q$;
       \If{$\scalar{a_k,y}_Q< -\varepsilon \|a_k\|_Q \|y\|_Q$}
          \State {\bf update}
          $\displaystyle x:=x-\frac{\scalar{a_k,y}_Q}{\|a_k\|_Q^2}\vec{e}_k$;\quad
           $\displaystyle y:=y-\frac{\scalar{a_k,y}_Q}{\|a_k\|_Q^2}a_k$;
         \Else~\Call{Rescale}{$Q,y$};

\State $T:=T\cup \{k\in S\sm T: \|\hat a_k\|_Q>\bart^{-1} \}$,\label{li:k too long}
\If{$\rk(A_{S\sm T})<\rk(A_S)$} \Call{Remove}{$T$};\label{li:rank}\EndIf
	\EndIf
    \EndWhile
 \If{$\Pi x> 0$}
\State Define $\bar x_i\in \R^n$ by $\bar x_i:=(\Pi x)_i/\|a_i\|_2$ if $i\in S$, $\bar x_i:=0$ if $i\notin S$.
 \Return{$\bar x$.}
    \EndIf
 \If{$S=\emptyset$}
\Return{$\bar x=0$.}
\EndIf
 \end{algorithmic}
\caption{Maximum Support Kernel Algorithm}\label{alg:primal-algorithm-norm}
\end{algorithm}

\begin{algorithm}[H]
\raggedright
  \begin{algorithmic}[1]\label{alg:remove}
\Procedure{Remove}{$T$}
    \State $S:=S\setminus T$; $T:=\emptyset$.
       \State Reset $x_j:=1$ for all $j\in S$;\quad $y:=\hat A_S x$.
       \State Recompute $\Pi:=\Pi_{\hat A_S}^K$;
\EndProcedure
 \end{algorithmic}
\caption{Column deletion}
\end{algorithm}

\end{minipage}
\end{center}
\end{figure}

Algorithm~\ref{alg:primal-algorithm-norm} terminates either with a solution $\bar x\in\ker(A)_+$ with $\supp(\bar x)=S$, in which case we may conclude $S=S^*$, or when $S=\emptyset$ is reached, in which case we may conclude that $\bar x=0$ is a maximum support solution.

\begin{theorem}\label{thm:complexity primal max support} Let $A\in\Z^{m\times n}$. Algorithm~\ref{alg:primal-algorithm-norm}  finds a solution of $Ax=0$, $x\geq 0$ of maximum support in $O\left((m^3n+mn^2)L\right)$ arithmetic operations.
\end{theorem}

The proof of  Theorem~\ref{thm:complexity primal max support} requires
the following lemma that gives a bound on the  volume increase of the
relevant ellipsoid at a column removal step.

\begin{lemma}\label{lemma:volume after removal} Consider a stage of
Algorithm~\ref{alg:primal-algorithm-norm} in which \Call{Remove}{$T$} is called.
Let $r=\rk(A_S)$, $r'=\rk(A_{S\sm T})$, and let $E:=E_{\im(A_S)}(Q)$, and
$E':=E_{\im(A_{S\sm T})}(Q)$, where $S,T$ are as in line~\ref{li:rank}. Then
\[\frac{\vol_{r'}(E')}{\vol_{r}(E)}\leq \frac{\nu_{r'}}{\nu_{r}}\left(\frac{2}{\bart^2}\right)^{r-r'} .\]
\end{lemma}
\begin{proof}
Let $T'$ denote the state of $T$ in the previous iteration, i.e.~before the update
at line~\ref{li:k too long}. Since \textproc{Remove} was not called in the
previous iteration, we have that $r = \rk(A_S) = \rk(A_{S \setminus T'}) >
\rk(A_{S \setminus T}) = r'$. Since rank can only decrease by one after the
removal of a column, we can construct a sequence of sets $S \setminus T :=
S_{r'} \subset S_{r'+1} \subset \cdots \subset S_r := S \setminus T'$ such that
$\rk(A_{S_k}) = k$ for $k \in \set{r',\dots,r}$. To prove the desired statement,
it suffices to show that for $k \in \set{r'+1,\dots,r}$,
\[
\frac{\vol_{k-1}(E_{k-1})}{\vol_{k}(E_{k})}\leq
\frac{\nu_{k-1}}{\nu_{k}} \left(\frac{2}{\bart^2}\right)~,
\]
where $E_l := E_{\im(A_{S_l})}(Q)$, $l \in \set{r',\dots,r}$. This follows by
induction, recalling that $E_{r'} = E'$ and $E = E_r$ (since
$\im(A_{S_r})=\im(A_{S \sm T'}) = \im(A_S)$).  

Let $k \in \set{r'+1,\dots,r}$, $H_{k-1} :=
\im(A_{S_{k-1}})$, $H_k := \im(A_{S_k})$. Let $\nu \in H_k$ be the vector
orthogonal to $H_{k-1}$ such that $\|\nu\|_2=1$. By Lemma~\ref{lemma:ellipsoid and hyperplane},
\[
\vol_{k-1}(E_{k-1})=\frac{\vol_{k}(E_k)}{\width_{E_k}(\nu)}\cdot\frac{\nu_{k-1}}{\nu_{k}},
\]
and hence it suffices to show that $\width_{E_k}(a) \geq \frac{\bart^2}{2}$.

Since at every rescaling the $Q$-norm of the columns of $\hat{A}_{S \setminus
T'}$ increases by at most a factor $2$ and since in the previous iteration they
had $Q$-norm at most $\theta^{-1}$ (otherwise they would have been added to
$T'$), their $Q$-norm during the current iteration is at most $2\theta^{-1}$. In
particular, since $S_k \subseteq S \setminus T'$, we have $\|\hat{a}_i\|_Q \leq
2\theta^{-1} ~ \forall i \in S_k$. From here, we have that
\begin{eqnarray*}
\width_{E_k}(\nu)&=&\max\{\nu^\T   z \st \|z\|_Q \leq 1,\, z\in H_k\}\\
&\geq & \max_{i \in S_k}\frac{|\nu^\T  \hat a_i|}{\|\hat a_i\|_Q}\geq
\frac{\bart}{2 }\min_{y \in H_k \sm \set{0}} \max_{i \in {S_k}}{|\hat y^\T  \hat a_i|}\\
&= & \frac{\bart}{2}|\rho_{(A_{S'} ,-A_{S'} )}|\geq \frac{\bart^2}{2},
\end{eqnarray*}
where the last inequality follows from  $|\rho_{(A_{S'} ,-A_{S'} )}|\geq \bart$.
\qued \end{proof}

\begin{proof}[Proof of Theorem~\ref{thm:complexity primal max support}]

If the algorithm terminates with $\Pi x>0$, then it correctly outputs a solution
to (\ref{eq:main problem}). Next we observe that, throughout the algorithm,
$S\supseteq S^*$, which implies that the solution returned at the end is always
a maximum-support solution. To prove this we only need to show that $T\subseteq
T^*$ throughout. New elements are added to $T$ in step \ref{li:k too long},
which are in $T^*$ by Lemma~\ref{lem:contained-in-ellipsoid} and the fact that
$\rho_{A_{S^*}}\geq \bart_A$ if $S\neq \emptyset$.

 
We need to argue that the algorithm terminates in the claimed number of iterations.  Recall that, by Lemma~\ref{lem:X-S} we have $P_A=P_{A_S}=P_{A_{S^*}}$ throughout the algorithm, since $S^*\subseteq S$.

A {\em round} of the algorithm consists of the iterations that take place between two consecutive calls of \textproc{Remove}. Since \Call{Remove}{$T$} is called only when $\rk(A_{S\sm T})<\rk(A_S)$, the number of rounds is at most $\rk(A)\leq m$.
We want to bound the total number of rescalings performed by the algorithm.

\begin{Claim}\label{cl:K-bound}
The total number $K$ of rescalings throughout the algorithm is $ O(m\log(\bart^{-1}))$.
\end{Claim}
\begin{proof}
In any given round, let $E:=E_{\im(A_S)}(Q)$ and  $r=\rk(A_S)$. We first show that at every rescaling within the round, except for the last, the invariant
\begin{equation}
\vol_r(E)\geq \nu_r\bart^{2r}\label{eq:volume-lower}
\end{equation}
is maintained.
Indeed, by Lemma~\ref{lem:contained-in-ellipsoid}(i), $P_A\subseteq E$ throughout. Since $\|\hat a_j\|_Q\leq \bart^{-1}$ for all $j\in S\sm T$, it follows that $E\supseteq \bart\conv(\hat A_{S\sm T}, -\hat A_{S\sm T})$. Since at every rescaling except for the last one of the round we have  $\rk(A_{S\sm T})=r$, it follows by Lemma~\ref{lem:rho} that $\conv(\hat A_{S\sm T}, -\hat A_{S\sm T})$ contains an $r$-dimensional ball of radius $|\rho_{(A_{S\sm T}, -A_{S\sm T})}|\geq \bart$. This implies \eqref{eq:volume-lower}.

At the first iteration, $Q=I_m$, $S=\emptyset$, and $E=\bb{B}_m\cap \im( A )$, therefore initially $\vol_{r}(E)\leq \nu_{r}$.
 By Lemma~\ref{lem:contained-in-ellipsoid}(iii), at every rescaling in which we do not remove any column $\vol_r(E)$ decreases by at least 2/3; Lemma~\ref{lemma:volume after removal} bounds the increase in $\vol_r(E)$ at column removals.
Combined with the lower bound \eqref{eq:volume-lower}, we obtain that the total number of  rescalings is at most $m$ plus the smallest number $K$ satisfying
\[\left(\frac23\right)^K \cdot\left(\frac{2}{\bart^2}\right)^m<\bart^{2m}.\]
The claimed bound on $K$ follows.
\qued \end{proof}

By Lemma~\ref{lem:delta-rounding}, the algorithm is guaranteed to terminate when $\|y\|_2<|\rho_{(A_S,-A_S)}|$, so in particular $\|y\|_2\geq \bart$ throughout the algorithm, since $\bart\leq|\rho_{(A_S,-A_S)}|$. By Lemma~\ref{lemma:rescaling stretch}, after $t$ rescalings $\|y\|_2\leq \|y\|_Q(1+3\varepsilon)^t$. Hence the algorithm is guaranteed to terminate if $\|y\|_Q\leq \bart/(1+3\varepsilon)^K$.

At the beginning of each round, we re-initialize $x$ so that $x_j=1$ for all $j\in S$. In particular,  $y=\hat A_S x$ satisfies $\|y\|_Q\leq {|S|}\bart^{-1} \leq {n}\bart^{-1} $, since
$\|\hat a_j\|_Q\leq \bart^{-1} $ for all $j\in S$.  At every rescaling within the same round, $\|y\|_Q$ increases by $2/(1+3\varepsilon)$, and in every DV update, it decreases by at least a factor $\sqrt{1-\varepsilon^2}$. Let $R$ be the number of rounds, and $K_1,\ldots,K_R$ be the number of rescaling within rounds $1,\ldots,R$.

It follows that, at the $i$th round, the number of DV updates is at most the smallest number $\kappa_i$ such that
\[{n}\bart^{-1}(1-\varepsilon^2)^{\kappa_i/2} 2^{K_i}< \bart/(1+3\varepsilon)^K.\]
Taking the logarithm on both sides and recalling that $\log(1-\varepsilon^2)<-\varepsilon^2$ and $\log(1+3\varepsilon)\geq 3\varepsilon$, it follows from our choice of $\varepsilon$ that $\kappa_i\in O(m^2) K_i+O(m)K$. Since $K=K_1+\cdots+K_R$ and $R\leq m$, this implies that the total number of DV updates is $O(m^2) K$. Using Claim~\ref{cl:K-bound}, the total number of DV updates is $O(m^3\log(\bart^{-1}))$. As explained in the proof of Theorem~\ref{thm:main-matrix}, we can perform each DV update in $O(n)$ arithmetic operations, provided that at each rescaling we recompute $F=A_S^\T Q A_S$, which as we showed can be done in $O(n^2)$ arithmetic operations. Observe also that $\|a_j\|^2_Q$ is the $j$th diagonal entry of $F$.  Since $\bart^{-1}\leq 2^{4L}$ by Lemma~\ref{lem:numerical-shit}, the total number of arithmetic operations performed for DV updates and rescalings is within the stated bound.

Every time a new column is added to $T$ at step \ref{li:k too long}, we need to then test at step~\ref{li:rank} is $\rk(A_{S\sm T})<\rk(A_S)$. This can be done in $O(m^2n)$ operations via Gaussian elimination. Since new columns are added to $T$ at most $n$ times, and since $n\leq L$, the total number of arithmetic operations required to test rank is $O(m^2n L)$, which is within the stated running time bound.

Finally, at the beginning of each round we need to recompute the projection matrix. Computing each projection matrix requires time $O(n^2m)$, and the total number of rounds is at most $m$. Since $n\leq L$, the total number of arithmetic operations performed to recompute the projection  matrices is $O(m^2nL)$, which is within the stated bound.
\qued \end{proof}

\subsection{Amortized Maximum Support Image Algorithm}\label{sec:app-image}
Analogously to the kernel setting, we now improve the running time of
the Basic Maximum Support Image Algorithm (Section~\ref{sec:image-max}) by a factor $m$.
The Maximum Support Image Algorithm (Algorithm
~\ref{alg:dual-alg-maxsup}) maintains a set $T$ of
indices with the property $T^*\subseteq T$. The set $T$ is
initialized as $T=[n]$, and we remove an index $a_k$ once we conclude
that $k\notin T^*$. The algorithm terminates with a
solution $\bar y$ such that $a^\T_k \bar y>0$ for all $i\in T$ and
$a^\T_k \bar y=0$
for all $i\notin T$, verifying $T=T^*$ at termination.
We maintain $r$ as the number of rows of $A$ throughout the algorithm. As in the full support case, we assume that initially the matrix has full row rank; this
will be preserved throughout the reduction steps.

\renewcommand{\algorithmicrequire}{\textbf{Input:}}
\renewcommand{\algorithmicensure}{\textbf{Output:}}
\begin{figure}[htb]
\begin{center}
\begin{minipage}{0.85\textwidth}
\begin{algorithm}[H]
\raggedright
  \begin{algorithmic}[1]
    \Require{A matrix $A\in\R^{m\times n}$ with $\rk(A)=m$.}
    \Ensure{A  solution $\bar  y\in \R^m$ to \eqref{eq:dual} satisfying the maximum number
    of strict inequalities.}
\State Compute $\bart:=\bart_A$ as in \eqref{def:delta}.
    \State Set $Q:=I_m$, $R:=I_m$, $U:=I_m$, $T:=[n]$, $r:=m$.
 \While{$T\neq\emptyset$}
    \State Call \Call{von Neumann}{$A,Q,\varepsilon$} to obtain $(x,y)$.\label{l:v-N}
    \If{$A^\T Qy>0$} \Return{$\bar y=UQy$.} {\bf Terminate.}\label{l:dual-feas}
\Else  {\bf\ rescale}
   \[R:=\frac{1}{1+\varepsilon}\left(R+\sum_{i\in T} \frac{x_i}{\|a_i\|_Q^2} a_ia_i^\T\right) ;\quad  Q:=R^{-1}.\]\label{l:rescale}
    \EndIf
 \While{$\exists k\in T$ such that ($\|\hat a_k\|_Q<\bart$) }
\State \Call{Remove}{$k$}.\label{l:short-column}
\EndWhile
 \EndWhile
\Return{$\bar y=0$.}
 \end{algorithmic}
\caption{Maximum Support Image Algorithm} \label{alg:dual-alg-maxsup}
\end{algorithm}

\begin{algorithm}[H]
\raggedright
  \begin{algorithmic}[1]\label{alg:remove-im}
\Procedure{Remove}{$k$}
 \State Select $W\in \R^{r\times (r-1)}$ whose columns form an orthonormal basis of $a_k^\bot$.
 \State Set $A:=W^\T A$, delete all $0$ columns, and remove the corresponding indices from $T$.
 \State Set $R:=W^\T RW$,  $U:=UW$, and $r:=r-1$. Recompute
 $Q=R^{-1}$.
\EndProcedure
 \end{algorithmic}
\caption{Column deletion}
\end{algorithm}
\end{minipage}
\end{center}
\end{figure}

\begin{theorem}\label{thm:complexity dual max support} Let the matrix $A\in
  \Z^{m\times n}$ have $\rk(A)=m$, and encoding length $L$.
Algorithm~\ref{alg:dual-alg-maxsup} finds a maximum support solution to
  $A^\T y\ge 0$ in
  $O\left(m^2n^2\cdot L\right)$  arithmetic
  operations. Using the Smoothed Perceptron of \cite{Pena-Soheili-smooth} or the  Mirror Prox method of \cite{Yu-Karzan-Carbonell} instead of von Neumann requires $O\left(m^3n\sqrt{\log n} \cdot L\right)$
  arithmetic operations.
\end{theorem}

We need the following stronger version of Lemma~\ref{lem:gamma-form}, with explicit
bounds on the coefficients. Note that the dimension $m$ is replaced by the
actual dimension $r$ and the set of columns $[n]$ by $T$.

\begin{lemma}\label{lem:gamma-form-2}
At any stage of the algorithm,  the matrix $R$ is positive definite and can be written in the form
\[
R=\alpha I_r+\sum_{i\in T} \gamma_i \hat{a}_i\hat{a}_i^\T
\]
where $\gamma_i \leq 2/\bart^2$, $\forall i \in T$,
$\alpha=1/(1+\varepsilon)^t$ for the total number of rescalings $t$ performed
thus far, and $\gamma_i\ge 0$. The trace is $\tr(R)=\alpha r+\sum_{i\in
T}\gamma_i$. Furthermore, for any
$v \in \R^r$ with $\|v\|=1$, we have $\|v\|_Q \geq \bart/{\sqrt{2(n+1)}}$.
\end{lemma}

\begin{proof}
Clearly any matrix of the form above is positive definite. The proof is by induction. The formula and bound are valid at initialization
when $R=I_m$ and $\gamma_i = 0 ~ \forall i \in [n]$. Let $R = \alpha I_r + \sum_{i
\in T} \gamma_i \hat{a}_i\hat{a}_i^\T$ denote the current decomposition, where
$\gamma_i \leq 2/\bart^2$. We show that the required form and bounds
hold for the next update.

Assume that we rescale in the current iteration. Let $R$ and $Q$
denote the matrices before, and $R'$ and $Q'$ after the rescaling.
For $i \in [n]$, using Lemma~\ref{lemma:ellipsoid width} we see that
\[
\|\hat{a}_i\|_Q^2 = \width^2_{E(R)}(\hat a_i)=\max \Big\{ ({\hat{a}_i}^\T{x})^2: \alpha \|x\|^2 + \sum_{j
\in T} \gamma_j (\hat{a}_j^\T{x})^2 \leq 1, x \in \R^r\Big\} \leq \frac1{\gamma_i} \text{.}
\]
Now, let $x$ be the convex combination returned by the von Neumann
algorithm in line~\ref{l:v-N}.  By the
rescaling formula in line~\ref{l:rescale}, the matrix $R$ is updated to $R'$ satisfying
\[
R' = \frac{1}{1+\epsilon}\Big(R + \sum_{i \in T} \frac{x_i}{\|a_i\|_Q^2} a_i a_i^\T\Big)
   = \frac{1}{1+\epsilon}\Big(\alpha I_r + \sum_{i \in T} \Big(\gamma_i +
\frac{x_i}{\|\hat a_i\|_Q^2}\Big) \hat{a}_i\hat{a}_i^\T\Big) \text{.}
\]
Hence, recalling that $\|\hat{a}_i\|_Q \geq \bart$ for every $i\in T$ at the beginning of every iteration, each $\gamma_i$ is updated to $\gamma_i'$ satisfying
\[
\gamma_i' = \frac{1}{1+\epsilon}\Big(\gamma_i + \frac{x_i}{\|\hat a_i\|_Q^2}\Big)
\leq \frac{2}{\|\hat a_i\|_Q^2} \leq \frac {2}{\bart^2}\text{.}
\]

Consider now a step when some columns are eliminated. Then the matrices $A$ and
$R$ are updated to $A'$ and $R'$, where $A'$ is obtained by removing the zero
columns from  $W^\top A$ and  $R'=W^\T RW$. We denote by $T'\subseteq T$ the
index set of columns of $A'$. Thus
\[
R'=\alpha W^\T W+\sum_{i\in T'} \gamma_i W^\T \hat{a}_i \hat{a}^\T_i W=\alpha
I_{r-1}+\sum_{i\in T'} \gamma_i\|W^\T \hat{a}_i\|^2 \frac{a'_i
{a'_i}^\T}{\|a'_i\|^2},
\]
where the last equality follows from $W^\T W=I_{r-1}$ and the fact that  $W^\T
a_i=0$ for all $i\in T\sm T'$. Setting $\alpha'=\alpha$ and $\gamma_i'=\gamma_i
\|W^\T \hat{a}_i\|^2$ for $i\in T'$ gives the desired decomposition of $R'$.
Next, since $\|W^\T\hat{a}_i\| \leq \|\hat{a}_i\| \leq 1$, we get that
$\gamma_i' \leq \gamma_i \leq 2/\bart^2$, for all $i \in T'$.

We now prove the last part lower bounding $\|v\|_Q$ for any unit
vector $v\in \R^r$. Firstly, for any $x\in \R^r$, the
Cauchy-Schwarz inequality gives
\[
x^\T R x = \alpha \|x\|^2 + \sum_{i \in T} \gamma_i ({\hat{a}_i}^\top{x})^2
\leq \Big(\alpha + \sum_{i \in T} \gamma_i\Big) \|x\|^2 \text{ ,}
\]
and hence $E(R)$ contains a Euclidean ball of radius at least $1/\sqrt{\alpha +
\sum_{i \in T} \gamma_i}$. Therefore, for any unit vector $v \in
\R^r$, using Lemma~\ref{lemma:ellipsoid width} we get
\[\|v\|_Q = \max \set{{v}^\top{x}: x \in E(R)}
\geq \frac{1}{\sqrt{\alpha + \sum_{i \in T} \gamma_i}}
\geq \frac{1}{\sqrt{1 + 2|T|/\bart^2}}
\geq \frac{\bart}{\sqrt{2(n+1)}} \text{ ,}
\]
as needed.
\qued \end{proof}

The next lemma gives a lower bound on the decrease in $\det(R)$ for
column removal steps.

\begin{lemma}\label{lem:remove-kosher}
Assume that at a given iteration $F_A \subseteq E(R)$, and consider an index $k\in T\setminus T^*$. Let $W\in\R^{r\times (r-1)}$ be a matrix whose columns form an orthonormal basis of $a^\bot_k$. Let $A'$ be the matrix obtained by removing all zero columns from $W^\T A$, and let $R'=W^\T RW$. Then  $F_{A'}\subseteq E(R')$ and $\det(R')\ge \det(R){\bart^2}/{(2(n+1))}$.
\end{lemma}

\begin{proof}
The first part of the statement follows from the fact that
$F_A\subseteq E_{a_k^\bot} (R)$. Further,
Lemma~\ref{lem:im-dim-red}(i) implies $F_{A'}=W^\top WF_{A'}=W^\top F_A$.  Thus, $F_{A'}=W^\T F_A\subseteq W^\T E _{a_k^\bot} (R)=W^\T W E(R')=E(R')$. For the second part, note that $\det(R')=\det_{a^\bot_k}(R)=\det(R)\|\hat a_k\|^2_Q$ using Lemma~\ref{lemma:determinant
and hyperplane}.  To obtain the desired
bound, we use the estimate $\|\hat a_k\|^2_Q \ge \bart^2/(2(n+1))$ from
Lemma~\ref{lem:gamma-form-2}, which holds since $\hat{a}_k$ is a unit vector.
\qued \end{proof}

We now prove Theorem~\ref{thm:complexity dual max support} based on these
lemmas and the results proved in Section~\ref{sec:dual-full-analysis}.

\begin{proof}[Proof of Theorem~\ref{thm:complexity dual max support}]
We first argue the correctness of the algorithm. Let $A$ be the input matrix, and let $A'$
be the current matrix during any stage of the algorithm. For the current matrix $R$, Lemmas~\ref{lem:F-in} and~\ref{lem:remove-kosher}
ensure that $F_{A'} \subseteq E(R)$. Therefore Lemma~\ref{lem:lower-width-2} implies that $T\supseteq T^*$ throughout the algorithm, so that $\Sigma_A$ is contained
in the subspace  $H:=\{x\in\R^m\st a_i^\T x=0\,\forall i\in[n]\sm T\}$. By construction, the columns of $U$ are an orthonormal basis of $H$,
and $A'$ is obtained from the matrix $U^\T A$ by removing the $0$ columns.

We next show that the solution $\bar y$ returned by the algorithm is
a solution to \eqref{eq:dual} satisfying the maximum number of strict
inequalities.  If $T=\emptyset$ at termination, then $\bar y=0$ is indeed a maximum support
solution. Assume the algorithm terminated at line~\ref{l:dual-feas} with $\bar
y=UQy$. Then, for every $k\in T$ we have $a_k^\T \bar y=a_k^\T UQy=(a_k')^\T Qy>0$, whereas
for every  $k\notin T$ we have $a_k^\T \bar y=a_k^\T UQy=0$ because $U^\T  a_k=0$ by construction.

We now prove that the algorithm terminates in the claimed  number of iterations.
Lemma~\ref{lem:a-k-Q-bound} remains valid; the proof uses
Lemma~\ref{lem:gamma-form-2} in place of Lemma~\ref{lem:gamma-form}. By Lemmas~\ref{lem:a-k-Q-bound} and~\ref{lem:im-dim-red}(iii), whenever $\det(R)> (1+\bart^{-2})^m$ we
can find $k\in T$ such that $\|a\|_Q<\bart$, thus we remove at least one column at step~\ref{l:short-column}. The potential $\det(R)$ is initially 1; by Lemma~\ref{lem:dual-vol-dec} it increases by at least a factor $16/9$ at every rescaling, and by Lemma~\ref{lem:remove-kosher} it decreases by at
most a factor $\frac{\bart^2}{2(n+1)}$ after the elimination of a column.
Since $\rk(A)$ decreases by $1$ every time we remove a column, the
algorithm performs at most $m$ column removals. Consequently, within
$O(m\log(n\bart^{-1}))=O(mL)$ rescalings, all columns outside $T^*$ will be
removed and the algorithm terminates.

As in the proof of Theorem~\ref{thm:main-matrix-dual}, the iterations
between two rescalings can be implemented in time $O(n^2m)$ whereas recomputing $R$ and $Q$ when rescaling requires $O(m^2n)$ operations. This contributes $O(m^2n^2L)$ to the overall running time. When removing a column, computing $W$ requires computing an orthonormal basis of $a_k$ in $\R^r$, which can be done by closed-form-formula in $O(r^2)$ arithmetic operations; computing $WA$ and $WRW$ require $O(m^2n)$ and $O(m^3)$, respectively; recomputing the inverse $Q$ or $R$ requires $O(m^3)$ operations. Hence the total number of arithmetic operations needed for the $O(m)$ column removals is $O(m^3 n)$.  This implies the stated running-time bound.

From the above, following the proof of Theorem~\ref{thm:main-matrix-smoothed} we obtain the running time
bound for the Smoothed Perceptron of \cite{Pena-Soheili-smooth} or the  Mirror Prox method of \cite{Yu-Karzan-Carbonell}.
\qued \end{proof}

\section{Missing proofs}\label{app:proofs}

\rhodef*
\begin{proof} Note that $\rho_A=\tau_{\hat A}$ as defined  in Lemma~\ref{lemma:radius} below, which shows that  $|\rho_A|$ is the distance of $0$ from the relative boundary of $\conv(A)$.
{\bf (i)} By Lemma~\ref{lemma:radius}(ii), $\rho_A<0$ if and only if $0$ is in the relative interior of $\conv(A)$, which is the case if and only if there exists $x>0$ such that $Ax=0$.

{\bf (ii)} For any $\bar y\in\Sigma_A$, $\|\bar y\|=1$, the distance between $\bar y$ and the hyperplane $\{y\st a_j^\T y=0\}$ ($j\in [n]$) is $\hat a_j^\T \bar y$, therefore $\min_{j\in[n]}\hat a_j^\T \bar y$ is the distance of $\bar y$ from the boundary of $\Sigma_A$, that is, the radius of the largest ball centered at $\bar y$ and contained in $\Sigma_A$. The statement now follows from the definition of $\rho_A$.
\qued \end{proof}

\begin{lemma}\label{lemma:radius} Let $A\in\R^{m\times n}$. Let $p$ be a point of minimum norm in the relative boundary of $\conv(A)$. Define \[\tau_A\eqdef\max_{y\in\im(A)\setminus\{0\}} \min_{z\in\conv(A)}z^\T   \hat y.\]
\begin{enumerate}[(i)]
\item If $0\notin \conv(A)$, then $\|p\|=\tau_A=\min_{j\in[n]} a^\T_j \hat p$.
\item If $0$ is in the relative interior of $\conv(A)$, then $p$ is in the relative interior of some facet of $\conv(A)$ and $\|p\|=-\tau_A=\max_{j\in[n]} a^\T_j \hat p$.
\end{enumerate}
\end{lemma}

\begin{proof} {\bf (i)} Assume $0\notin\conv(A)$. Then $p$ is a point of minimum norm in $\conv(A)$. It follows that $p^\T z\geq \|p\|^2$ for every $z\in\conv(A)$, implying that $\|p\|\leq \tau_A$.
We now show that $\tau_A\leq \|p\|$. If not, then there exists $y\in\im(A)$ such that $\|y\|=1$ and $\min_{j\in [n]} a_j^\T y>\|p\|$. In particular, this implies that every point in $\conv(A)$ has distance greater than $\|p\|$ from the origin, contradicting our choice of $p\in\conv(A)$.
\medskip

{\bf (ii)}  Assume $0$ is in the relative interior of $\conv(A)$.
By our choice of $p$, for any $y\in\im(A)$, $\|y\|=1$, we have $z:=-\|p\|y\in\conv(A)$ and $z^\T y=-\|p\|$,
which implies that  $\tau_A\leq -\|p\|$.
For the other direction, consider any facet $F$ of $\conv(A)$ containing $p$, let $H$ be the affine hyperplane of $\im(A)$ containing $F$,
and let $q$ be the minimum norm point in $H$.
Since $p\in F\subseteq H$,  by definition $\|q\|\leq \|p\|$.
Since $F$ is a facet,  $q^\T z\leq \|q\|^2$ is a defining inequality for $F$ (i.e. it is verified by all $z\in \conv(A)$
and satisfied as equality by all $z\in F$). If we let $y=-q$, this shows $\tau_A\geq \min_{z\in\conv(A)}\hat y z\geq -\|q\|\geq -\|p\|$.
This shows that $p=q$ and $\tau_A=-\|p\|$. In particular, $F$ must be the only facet containing $p$, therefore $p$ is in the relative interior of $F$.
\qued \end{proof}

\begin{Claim}\label{lemma:width estimate}
Let $A\in\Z^{m\times n}$. If $T_A^*\neq\emptyset$, then
\[\displaystyle\max_{y\in \Sigma_A\sm\{0\}}\min_{j\in T^*} a_j^\T \hat y\geq \frac{1}{m^2 \Delta_A}.\]
\end{Claim}
\begin{proof}
Let $S^*:=S_A^*$ and $T^*:=T^*_A$. Let $(y^*,s^*)\in \R^{2m}$ be an optimal basic solution of the following linear program:
\begin{equation}\label{eq:1-norm y}
\begin{array}{rrcl}
\min & \vec{e}^\T s\\
&A^\T_{T^*}y&\geq &\vec{e}\\
&A^\T_{S^*}y &=&0\\
&s-y&\geq &0\\
&s+y&\geq &0
\end{array}
\end{equation}
This LP is feasible by the definition of $T^*$ and the optimal value
equals $\|y^*\|_1$.
Note that, for every square submatrix $A'$ of $A$, $|\det(A')|\leq \Delta_{A'}\leq \Delta_A$, where the first inequality follows from Hadamard's bound,
and the second from the fact that the entries of $A$ are integer.
From this fact, a straightforward application of Cramer's rule implies that  $s^*_j\leq m\Delta_A$ for all $j\in [n]$,
so $\|y^*\|_1\leq m^2\Delta_A$. Since by construction $y^*\in \Sigma_A\sm\{0\}$,
the statement follows from the fact that
\[ \min_{j\in T^*} a_j^\T \frac{y^*}{\|y^*\|_2}\geq\frac{1}{\|y^*\|_2}\geq \frac{1}{\|y^*\|_1}\geq \frac{1}{m^2 \Delta_A}.\]
\end{proof}

\numerical*

\begin{proof}
Let $\alpha:=\max_{i\in T^*}\|a_i\|$. Note that $|\rho_A|\geq |\tau_A|/\alpha$, where $\tau_A$ is defined as in Lemma~\ref{lemma:radius}. Since $\alpha\leq \Delta_A$, it suffices to show that $|\tau_A|\geq 1/(m^2\Delta_A)$.

If $0\notin\conv(A)$, then $T^*=[n]$ and we observe that $\tau_A=\displaystyle\max_{y\in \Sigma_A\sm\{0\}}\min_{j\in [n]} a_j^\T \hat y\geq \frac{1}{m^2 \Delta_A}$, where the inequality follows from  Claim~\ref{lemma:width estimate}. Assume now that $0$ is in the relative interior of $\conv(A)$. Let $p$ be a point of minimum norm in the relative boundary of $\conv(A)$. According to Lemma~\ref{lemma:radius}, $|\tau_A|=\|p\|$, and $p$ is contained in the relative interior of a facet $F$ of $\conv(A)$. Let $A'$ be the submatrix of $A$ comprised of the columns that are contained in $F$. In particular, $\conv(A')=F$, therefore $0\notin\conv(A')$, which implies $\tau_{A'}>0$. By the previous argument, $\tau_{A'}\geq 1/(m^2\Delta_{A'})\geq 1/(m^2\Delta_A)$. Since $p$ is the point of minimum norm in $F$, it follows from Lemma~\ref{lemma:radius} that $\tau_{A'}=\|p\|$. It follows that $|\tau_A|=\tau_{A'}\geq 1/(m^2\Delta_{A'})$.

Finally, since $\Delta_A\leq 2^L$ (see for example \cite[Lemma 1.3.3]{glsbook}) and $m\leq L$, it follows that $1/(m^2\Delta^2_A)\geq 2^{-4L}$.
\qued \end{proof}

\symmetrized*
\begin{proof}
We first show $P_A=P_{A_{S^*}}$. The inclusion $P_{A_{S^*}}\subseteq P_A$ is obvious. For the reverse inclusion, consider $y\in P_A$ and let $x,z\in\R^n_+$ such that $\vec{e}^\T x=\vec{e}^\T z=1$ and $y=\hat Ax=-\hat Az$. Then $\hat A(x+z)=0$, $x+z\geq 0$, which implies $x_i=z_i=0$ for all $i\in[n]\sm {S^*}$, which shows that $y\in P_{A_{S^*}}$.

We show $\spn(P_{A})=\im(A_{S^*})$. It suffices to show that
$\spn(P_{A_{S^*}})=\im(A_{S^*})$ because $P_A=P_{A_{S^*}}$. The inclusion
$\spn(P_{A_{S^*}})\subseteq \im(A_{S^*})$  is obvious. For the reverse inclusion, it suffices to show that, for every $i\in S$, there exists $\alpha\neq 0$ such that $\alpha a_i\in P_{A_{S^*}}$. Consider $\lambda\in\R^{|{S^*}|}_{++}$ such that $\hat A_{S^*} \lambda=0$, and assume without loss of generality that $\sum_{j\in {S^*}\sm \{i\}}\lambda_j=1$. Then $-\lambda_i \hat a_i=\sum_{j\in {S^*}\sm \{i\}}\lambda_j \hat a_j$, which implies $-\lambda_i \hat a_i\in P_{A_{S^*}}$.
\qued \end{proof}

\omegab*
\begin{proof}
First, observe that \[\omega_A=\min_{j\in T^*} \max_{y\in \Sigma_A\sm\{0\}}\hat a_j^\T \hat y\geq \max_{y\in \Sigma_A\sm\{0\}}\min_{j\in T^*}\hat a_j^\T \hat y\eqdef\eta_A.\]
Note that, if $T^*=[n]$, then $\eta_A=\rho_A$, which proves the first part of the statement. For the second part of the statement, assume that $A$ has integer entries and that $T^*\neq\emptyset$. Letting $\alpha:=\max_{i\in T^*}\|a_i\|$, we have
\[\eta_A\geq \alpha^{-1} \max_{y\in \Sigma_A\sm\{0\}}\min_{j\in T^*} a_j^\T \hat y\geq \frac{1}{m^2 \Delta^2_A}=\theta_A,\]
where the last inequality follows from $\alpha\leq \Delta_A$ and Claim~\ref{lemma:width estimate}.
\end{proof}


\end{document}